\documentclass{amsart}

\usepackage{amssymb}

\usepackage{enumerate}
\usepackage{color}

\newcommand\eps{\epsilon}

\newcommand\id{{\operatorname{Id}}}
\newcommand\lessfine{\operatorname{\preceq}}

\newcommand\NN{{\mathbb N}}
\newcommand\Pena{{\mathbb P}'_{\operatorname{erg}}}
\newcommand\pen{\Pena}
\newcommand\Prob{{\mathbb P}}
\newcommand\RR{{\mathbb R}}
\newcommand\sepa{{\vert\,}}
\newcommand\spa{{\emptyset}}
\newcommand\topo{{\operatorname{top}}}
\newcommand\ZZ{{\mathbb Z}}
\newcommand\mme{m.m.e.}

\newcommand{\epm}{entropy-period-maximal}

\newtheorem{theorem}{Theorem}[section]
\newtheorem{proposition}[theorem]{Proposition}
\newtheorem{lemma}[theorem]{Lemma}
\newtheorem{corollary}[theorem]{Corollary}
\newtheorem{fact}[theorem]{Fact}

\theoremstyle{definition}
\newtheorem{definition}[theorem]{Definition}
\newtheorem{example}[theorem]{Example}
\newtheorem{xca}[theorem]{Exercise}

\theoremstyle{remark}
\newtheorem{remark}[theorem]{Remark}

\numberwithin{equation}{section}

\newcommand\alt[1]{\left\{\begin{array}{cl}#1\end{array}\right.}

\newcommand{\Proberg}{{\Prob_\erg}}
\newcommand{\erg}{{\operatorname{erg}}}
\newcommand{\notfiner}{{\preceq}}
\newcommand{\per}{{\operatorname{per}}}

\begin{document}

\title[Almost Borel structure]{The almost Borel structure of diffeomorphisms with some hyperbolicity}


\author{J\'er\^ome Buzzi} 
\address{Laboratoire de Math\'ematiques d'Orsay (CNRS \& UMR 8628), B\^at. 425, Universit\'e Paris-Sud, 91405 Orsay France}
\email{jerome.buzzi@math.u-psud.fr}
\thanks{The author gratefully acknowledges the support of the semester Hyperbolicity, large deviations and fluctuations organized at the {\it Centre Interfacultaire Bernoulli} at {\it \'Ecole polytechnique f\'ed\'erale de Lausanne} where a version of these lectures was delivered}



\date{Compiled on \today}

\begin{abstract}
These lectures focus on a recent result of Mike Hochman: an arbitrary standard Borel system can be embedded into a mixing Markov with equal entropy, respecting \emph{all invariant probability measures}, with two exceptions: those carried by periodic orbits and those with maximal entropy. We discuss the corresponding notions of almost Borel embedding and isomorphism and universality.

The main part of this paper is devoted to a self-contained and detailed proof of Hochman's theorem.  We then explain how Katok's horseshoe theorem can be used to analyze diffeomorphisms with "enough" measures that are hyperbolic in the sense of Pesin theory, in both mixing and non-mixing situations. In the latter setting, new invariants generalizing the measures maximizing the entropy emerge.
\end{abstract}

\maketitle

\setcounter{tocdepth}{1}
\tableofcontents


\section{Introduction}

Ornstein's classical theory \cite{Ornstein} gave powerful criteria to show that many naturally occuring measure-preserving transformations are isomorphic to Bernoulli schemes and are completely classified by their entropy. Since then, similar classifications by entropy have been obtained in more rigid categories. One of the first such results is the classification by Adler and Marcus \cite{AM} of mixing shifts of finite type up to almost topological conjugacy. The goal of these lectures is to explain a recent, striking extension of this circle of ideas achieved by Mike Hochman \cite{Hochman}. We will give an essentially self-contained proof and some applications to smooth dynamics, mainly based on Katok's theorem on the approximation of hyperbolic measures by horseshoes \cite{KatokIHES}.

\subsection{Hochman's Theorem}

The subject of these lectures is the following:\footnote{We refer to Sec. \ref{sec.back} for notations, definitions and background.}

\begin{theorem}[Hochman \cite{Hochman}]\label{thm.Hochman1}
Let $(\Sigma,\sigma)$ be a mixing Markov shift with entropy $h(\Sigma)$. Given any standard Borel dynamical system (or Borel system, for short) $(X,S)$, let $\Pena(S)$ be the set of its aperiodic, ergodic invariant probability measures.

Then any Borel dynamical system $(X,S)$ such that:
 $$
 (*)\qquad\qquad \forall\mu\in\Pena(S) \quad h(S,\mu)<h(\Sigma)\qquad
 $$
has an \emph{almost Borel embedding} in $\Sigma$, i.e., there is a map $\psi:X'\to \Sigma$ satisfying:
 \begin{enumerate}
 \item $X'$ is Borel with  $\mu(X')=1$ for any $\mu\in\Pena(S)$;
 \item $\psi:X'\to\Sigma$ is Borel;
 \item $\psi\circ S=\sigma\circ\psi$ over $X'$. 
 \end{enumerate}
In other words, $\Sigma$ is \emph{almost Borel universal} in the class of standard Borel systems satisfying (*).
\end{theorem}

We discuss this striking result in Sec. \ref{sec.almostBorel}. This is a Borel version of the following Krieger's generator theorem \cite{Krieger,Krieger2}\footnote{There are many versions of this classical theorem. Another one, perhaps closer to Hochman's theorem, can be found in \cite[chap. 31]{DGS}: given a mixing SFT $\Sigma$, if $(S,\mu)$ is an aperiodic measure-preserving system and if for almost all ergodic components $\nu$ of $\mu$, $h(S,\nu)<h_\topo(\Sigma)$, then there is a measure-preserving embedding of $S$ into $\Sigma$ whose image is topologically minimal and uniquely ergodic. It is also interesting to compare with \cite{BeguinCrovisierLeRoux} which considers homeomorphisms of manifolds, not Cantor spaces.}:

\begin{theorem}[Krieger's generator theorem]\label{t.Krieger}
Let $(\Sigma,\sigma)$ be a mixing SFT with entropy $h_\topo(\Sigma)$. Let  $(S,\mu)$ be an ergodic system which is not reduced to a periodic orbit.
If $h(S,\mu)<h(\Sigma)$ then there is a measurable embedding of a full measure subset of $(S,\mu)$ into $\Sigma$.
\end{theorem}

This theorem of Krieger already implied that a mixing Markov shift contains any aperiodic and ergodic system of smaller entropy. A key difference is that in Hochman's theorem both the system to be embedded and the target system are of the same (Borel) nature. This allows a Cantor-Bernstein principle (see Lemma \ref{l.DCB}): mutual almost Borel embeddings imply isomorphism in the following sense.

\begin{definition}
Two Borel systems $(X,S)$, $(Y,T)$ are \emph{almost-Borel isomorphic} if there exists a Borel isomorphism $\psi:X'\to Y'$ such that:
 \begin{enumerate}
 \item $X',Y'$ are Borel; $\mu(X')=\nu(Y')=1$ for any $\mu\in\Pena(S)$, $\nu\in\Pena(T)$;
 \item $\psi\circ S=T\circ\psi$ over $X'$. 
 \end{enumerate}
Sets satifying the property (1)  are said to be \emph{almost all} of the Borel system.
\end{definition}

From these results, Hochman obtained a classification of mixing Markov shifts by their entropy and their possession (or not) of a \mme\ He then easily deduced that many natural systems are isomorphic to Markov shifts.

In Sec.\ \ref{sec.almostBorel}, we will discuss the interpretation of Theorem \ref{thm.Hochman1} in terms of universality, how it has been applied by Hochman to systems containing "enough" embedded mixing SFTs and finally, compare almost Borel isomorphism with related notions.

\subsection{A proof}

The main part of these lectures is devoted to a self-contained proof of Hoch\-man's result. We give all necessary definitions and background and rely only on basic results like the Kuratowski theorem from descriptive set theory or the Shannon-McMillan-Breiman theorem from ergodic theory. 
We essentially follow Hochman's ideas, with only minor technical simplifications or variations (e.g., we do not use B. Weiss countable generator theorem or the finitary coding, and we use a Borel construction of a Borel subset of given measure).

The first step of Hochman's proof establishes a Borel version of Krieger's Theorem \ref{t.Krieger} with embedding into some special mixing SFT and coding of the measure by a binary sequence. 

\begin{theorem}[See Theorem \ref{t.BorelKrieger}]
For any  Borel system $(X,S)$ and $\eps>0$, there are a mixing SFT $(\Sigma,\sigma)$ with $h(X)<h_\topo(\Sigma)<h(X)+\eps$ and Borel maps $\phi:X'\to\{0,1\}^{\NN}$ and $\psi:X'\to\Sigma$ with $X'$ almost all of $X$ such that: $\phi\circ S=\phi$, $\psi\circ S=\sigma\circ\psi$, and $(\phi\times\psi):X'\to \{0,1\}^{\NN}\times\Sigma$ is injective.
\end{theorem}

The second step builds another equivariant Borel map $\Psi$ into  another SFT such that $\Psi(x)$ determines both $\phi(x)$ and $\psi(x)$, hence is injective. This $\Psi(x)$ is built by "splicing" into $\psi(x)$ an equivariant version of $\phi(x)$ (obtained by considering the times of visit to a well-chosen set).

\begin{theorem}[See Theorem \ref{t.restricted}]
For any  Borel system $(X,S)$ and $\eps>0$, there are a mixing SFT $\tilde\Sigma$ with $h(X)<h_\topo(\tilde\Sigma)<h(X)+\eps$ and an almost Borel embedding $\Psi:X\to\tilde\Sigma$.
\end{theorem}

To conclude, one embeds the image of $\Psi$ into the given mixing SFT (Lemma \ref{l.embedSigma}) using markers (Lemma \ref{l.marker}) and then removes  auxiliary assumptions which simplified the previous steps: the target Markov shift does not have to be an SFT, the entropy inequality (*) does not imply a uniform entropy gap.

\subsection{Application to smooth dynamics with mixing}
In the rest of these lectures, we consider $C^{1+}$-diffeomorphisms $T$ of compact manifolds, i.e., $C^1$-diffeo\-mor\-phisms whose differential $T'(x)$ is a H\"older-continuous function of $x$.\footnote{This smoothness assumption is required by the proof of Katok's theorem which relies on Pesin theory.}
We use the classical approximation result of A. Katok \cite{KatokIHES} (see Thm.\ \ref{t.KatokPeriod} and Sec.\ \ref{sec.back.hyp}): any ergodic, invariant probability measure which is hyperbolic can be approximated by horseshoes. In Sec.\ \ref{sec.Katok}, we recall Katok's theorem and relate the period of the horseshoe with the periods of the measure, a piece of information which is necessary for our purposes.

Perhaps unexpectedly, Hochman's theorem turns such approximations into isomorphisms. In fact, Katok's theorem provides the embedded mixing SFTs needed to follow Hochman's approach. This shows that diffeomorphisms with "enough"  hyperbolic measures (ie, without zero Lyapunov exponents) are almost Borel isomorphic to Markov shifts up to measures of maximal entropy.

Our first results assume some mixing.
Recall that a measure-preserving system $(T,\mu)$ is \emph{totally ergodic} if all its iterates are ergodic and it is \emph{$p$-Bernoulli} if it is isomorphic to the product of a Bernoulli system (see Sec. \ref{sec.back.shift}) and a circular permutation on $p$ elements. 

\begin{theorem}\label{t.mixdiff}
Let $T$ be a $C^{1+}$-diffeomorphism of  a compact manifold $M$. Assume that: {\rm(\#)} for any $h<h_\topo(T)$, there is a totally ergodic, hyperbolic invariant measure with entropy $>h$. 
Then, 
\begin{enumerate}
 \item $(M,T)$ is a disjoint union of Borel subsystems $M_1\sqcup M_2$ such that $T|M_1$ is almost Borel isomorphic to a mixing Markov shift and $M_2$ carries exactly the measures of maximal entropy of $T$;
 \item $(M,T)$ is almost Borel isomorphic to a Markov shift if and only if: {\rm(\S)} $T$ has at most countably many \mme\ (i.e., ergodic measures maximizing the entropy) and each one is $p$-Bernoulli for some $p\geq1$. 
\end{enumerate}
\end{theorem}

For surface diffeomorphisms, results of Sarig \cite{Sarig} and Berger \cite{Berger} imply the condition (\S) in point (2) of the above theorem and we obtain:

\begin{corollary}\label{c.sarig}
Let $T$ be a $C^{1+}$-diffeomorphism of a compact surface with positive entropy and a totally ergodic \mme\ Then $T$ is almost Borel isomorphic to a Markov shift. 

Moreover, such diffeomorphisms are classified up to almost Borel isomorphism by the following data: their topological entropy and the (possibly zero or infinite) number of their \mme's that are $p$-Bernoulli for each $p\geq1$. 
\end{corollary}

\begin{corollary}\label{c.berger}
Consider H\'enon-like maps $H_{a,b}$  where $(a,b)$ belongs to a good set of parameters (see \cite{Berger} for precise definitions). Each such map is almost Borel isomorphic to any positive recurrent mixing Markov shift with entropy $h(H_{a,b})$. Moreover, these maps are classified up to almost Borel isomorphism by their entropy. 
\end{corollary}

In \cite{BF}, we considered diffeomorphisms of the type introduced by Bonnati and Viana \cite{BV}. As announced at the end of Sec. 1.3 of that paper:
 
\begin{corollary}\label{c.BF}
The robustly transitive, non-partially hyperbolic diffeomorphisms obtained in \cite{BF} by deformation of Anosov diffeomorphisms satisfy the following stability property. Any $C^1$-pertubation of  such a diffeomorphism is almost Borel isomorphic to the initial Anosov diffeomorphism.
\end{corollary} 

\subsection{Application to smooth dynamics without mixing}
Finally, in section \ref{sec.notmixing}, we remove the assumption of mixing. We use some general tools developed in \cite{ABS}. It turns out that one has to take into account entropy, not only globally, but "at given period" (see Sec. \ref{sec.periods} for the periods of an ergodic system). This involves the following generalization of \mme's:

\begin{definition}\label{d.epm}
Let $(M,T)$ be a Borel system. A measure $\mu\in\Pena(T)$ is \emph{\epm} if for any measure $\nu\in\pen(T)$ with set of periods $\per(T,\nu)\subset\per(T,\mu)$, one has $h(T,\nu)\leq h(T,\mu)$.
\end{definition}

\begin{remark}\label{rem.epmHyp}
It follows from Katok's theorem that, if $T$ is a $C^{1+}$-diffeo\-mor\-phism of a compact manifold, any \epm\ measure $\mu\in\pen(T)$ is hyperbolic unless, possibly, if it has zero entropy.
\end{remark}

We show:

\begin{theorem}\label{t.hypdiff}
Let $T:M\to M$ be a $C^{1+}$-diffeo\-mor\-phism  of a compact manifold $M$. Assume that there is $h_0<h_\topo(T)$ such that all ergodic measures with entropy $>h_0$ are hyperbolic.

Then $T$ is the disjoint union  of three Borel subsystems $M_0\sqcup M_1\sqcup M_2$ such that:
 \begin{enumerate}
  \item $M_0$ carries only non hyperbolic ergodic measures with entropy $<h_0$;
  \item $M_1$ is almost Borel isomorphic to a Markov shift;
  \item $M_2$ carries only \epm\ measures of $(M,T)$.
 \end{enumerate}
\end{theorem}

The entropy-hyperbolic condition studied in \cite{BuzziICMP} defines non-empty open sets of $C^\infty$-diffeomorphisms such that the measures in $M_0$, i.e., the non hyperbolic measures, have entropy $\leq h$ for some $h<h_\topo(T)$. Hence the above theorem yields an entropy-conjugacy in the sense of \cite{BuzziSIM} (see  Sec. \ref{sec.partial}).

\subsection{Some open problems}

The application of Hochman's theorem to smooth ergodic theory strengthens in a spectacular way some previous results that dealt only with \mme's.\ It is perhaps even more interesting that it points to new invariants, like the \epm\ measures. It also asks new questions in smooth dynamics. Let us list three of them.

\subsubsection*{Measures maximizing the entropy}
The \mme's and \epm\ measures that appear in the above theorems cannot be analyzed by the techniques of this paper.
We analyze them in \cite{ABS} in the case of surface diffeomorphisms. But that work relies heavily on Sarig's symbolic dynamics \cite{Sarig} and the introduction and analysis of a "Bowen property". Can this be generalized, say to higher dimensions or partially hyperbolic systems?

\subsubsection*{Period and maximal entropy}
F. Rodriguez-Hertz, M. Rodriguez-Hertz, Tahzibi and Ures \cite{RHRHTU} have studied the \mme's of a class of partially hyperbolic systems. In the generic case, these \mme's are hyperbolic and periodic-Bernoulli. However, their periods can be larger than $1$. Is it possible that measures with smaller entropy have smaller period sets than any \mme?

\subsubsection*{Abundance of hyperbolicity}
We deduce universality from Katok's horseshoe theorem. Hence we need "enough" hyperbolic measures\footnote{However, see \cite{QS1,QS2} for measure-preserving universality without hyperbolicity.}: (*) any ergodic measure which is not \epm\ is dominated (Def. \ref{d.epdom}) by some hyperbolic measure. 
The usual tools to perturb Lyapunov exponents away from $0$ consider a nice but fixed invariant measure (the volume). Among the partially hyperbolic diffeomorphisms with central dimension $1$, are those with "enough" hyperbolic measure in the sense of (*) $C^1$ or $C^2$ dense?

\section{Definitions and background}\label{sec.back}

We recall some standard facts to make these lectures as self-contained as it is reasonable and to fix notations. We also prove some basic facts for which we did not find references.

\subsection{Borel spaces}
A \emph{standard Borel space} is $(X,\mathcal B)$ such that there exists a distance on $X$ such that $(X,d)$ is complete and separable, and $\mathcal B$ is the $\sigma$-algebra generated by the open sets. One often omits the Borel structure $\mathcal X$ if it is clear from the context. We refer to \cite{Kechris} for background.

\begin{example}
The following are standard Borel spaces:
$\NN$; $\RR$; $\NN^\NN$; compact metric spaces.
\end{example}

A map between two Borel spaces is \emph{Borel map} if the preimage of any Borel set is Borel. A \emph{Borel isomorphism} beween spaces is a Borel map, which is invertible and with Borel inverse.

\medbreak

The following result shows that standard Borel spaces are rather nice spaces.

\begin{theorem}[Kuratowski]\label{thm.Kuratowski}
All uncountable standard Borel spaces are pairwise isomorphic.
\end{theorem}



The direct image of a Borel subset is not always Borel. However (see \cite[(15.2)]{Kechris}):

\begin{theorem}[Lusin-Souslin]\label{thm.LusinSouslin}
Let $X,Y$ be standard Borel spaces. If $\Psi:X\to Y$ is a Borel injection, then $\Psi(X)$ is a Borel subset of $Y$ and $\Psi:X\to\Psi(X)$ is a Borel isomorphism. 

More generally, if $\Psi^{-1}(y)$ is countable for each $y\in Y$, then $\Psi$ has a \emph{Borel section}: $\Psi(X)$ is Borel and there exists a Borel map $\Phi:\Psi(X)\to X$ such that $\Phi\circ\Psi=\id_X$.
\end{theorem}

The following constructions stay within standard Borel spaces.

\begin{proposition}\label{p.stdBorelCons}
Let $(X,\mathcal X)$ be a standard Borel space. Then the following are also standard Borel spaces:
 \begin{enumerate}
 \item  If $(X_i)_{i\in I}$ is a countable family of standard Borel spaces, then so is their product $\prod_{i\in I} X_i$.
 \item any Borel subset $Y\in\mathcal X$ equipped with $\mathcal Y:=\{B\in\mathcal X:B\subset Y\}$;
 \item $\Prob(X)$, the set of Borel probability measures, equipped with the $\sigma$-algebra generated by $\mu\mapsto \mu(B)$, $B\in\mathcal X$;
 \item $\Prob(S)$ the subset of invariant probability measures;
 \item $\Proberg(S)$ the subset of ergodic, invariant probability measures. 
\end{enumerate}
Moreover, the above Borel structure of $\Prob(X)$ coincides with that coming from the usual  the weak star topology.
\end{proposition}

\subsection{Categories of dynamical systems}

In this paper, we mainly consider \emph{Borel systems}, i.e., $(X,\mathcal X,T)$ (or simply $(X,T)$) is a Borel automorphism $T$ of a standard Borel space $(X,\mathcal X)$. A \emph{Borel homomorphism} between two Borel systems $(X,S)$ and $(Y,T)$ is a Borel map $\psi:(X,\mathcal X)\to(Y,\mathcal Y)$ such that $\psi\circ S=T\circ\psi$. A \emph{Borel isomorphism} between two Borel systems $(X,S)$ and $(Y,T)$ is a Borel homomorphism  which is an isomorphism between the Borel spaces. A \emph{Borel embedding} of one system $(X,S)$ into another $(Y,T)$ is a Borel injective map $\psi:(X,\mathcal X)\to(Y,\mathcal Y)$ such that $\psi\circ S=T\circ\psi$. 

 We turn $(X,S)$ into a \emph{measure-preserving (dynamical) system} $(S,\mu)$ by selecting a measure $\mu\in\Prob(S)$. An \emph{ergodic system} is a measure preserving system which is ergodic. A measure-preserving system is \emph{aperiodic} if the set of periodic points has zero measure. We refer to the first chapters of, e.g., \cite{Petersen} for background on ergodic theory.
 
A property holds \emph{for almost all $x\in X$} if it holds for all $x$ off an \emph{almost null set}, i.e., it fails for a set of zero measure with respect to any aperiodic, ergodic  measure  (recall that \lq measure\rq\  means invariant probability measure, unless specified otherwise). Equivalently, the complement set is almost all.

An \emph{almost Borel map} $(X,S)\to(Y,T)$ is a Borel map defined on almost all of $X$. Almost Borel homomorphisms, embeddings and isomorphisms are defined in the obvious way. Recall that $\Pena(S)$ is the set of all aperiodic, ergodic measures.

\subsection{Entropy of dynamical systems}

We refer to chapters 5 and 6 of \cite{Petersen} for the following facts and to \cite{Downarowicz} for background. We consider a Borel system $(X,T)$.

A \emph{partition} $P$ is a countable Borel partition of $X$. The \emph{join} $\bigvee_{i\in I} P_i$ of a family $(P_i)_{i\in i}$ of partitions is the coarsest partition finer then any in the family. In particular, $P^{T,n}:=\bigvee_{k=0}^n T^{-k}P$. 
If $x\in X$ and $P$ is a partition of $X$, then $P(x)$ denotes the unique element of $P$ that contains $x$. For a partition $P$ and a subset $I\subset\ZZ$, the \emph{$P,I$-name} of a point $x\in X$ is the map $w:I\to P$ such that $w(i)$ is the element of $P$ containing $S^ix$. 

The \emph{Kolmogorov-Sinai entropy} of a measure-preserving system $(T,\mu)$ is:
 $$
    h(T,\mu) := \sup \{h(T,P,\mu):P\text{ partition of }X\}
 $$
where $
    h(T,P,\mu) ;= \lim_{n\to\infty}\frac1n H(P^{T,n},\mu)$ 
    with 
    $H(P,\mu) := \sum_{A\in P} -\mu(A)\ln\mu(A)\in[0,\infty]
 $
($0\ln0=0$).
$h(T,\mu)$ is an invariant of measure-preser\-ving isomorphism (i.e. isomorphism of Borel subsystems defined by full measure subsets).

Let $P_1,P_2,\dots$ be a refining sequence of partitions (each element of $P_{n+1}$ is contained in an element of $P_n$). It is generating with respect to $(T,\mu)$ if there is $X'\subset X$ with $\mu(X')=1$ such that for all $x,y\in X'$, $(\forall k\in\ZZ\forall n\in\NN^*\; P_n(T^kx)=P_n(T^ky))\implies x=y$. By Sinai's theorem, in this situation,
 $$
    h(T,\mu) = \lim_{n\to\infty} h(T,\mu,P_n).
 $$

\begin{xca}
Let $\Sigma=\{0,1\}^\ZZ$ with the shift $\sigma$.
Let $p(0)\in[0,1]$, $p(1):=1-p(0)$ and let $\nu_p$ be the unique $\sigma$-invariant Borel probability measure on $\Sigma$ such that $\mu([x_0\dots x_{n-1}]_m)=p(x_0)p(x_1)\dots p(x_{n-1})$. Show that $\{[0]_0,[1]_0\}$ is a generating partition and that $h(\sigma,\nu_p)=-p\log p-(1-p)\log p$.
\end{xca}

The following is convenient, if not standard. The supremum can be taken over all measures, without changing $h(T)$.

\begin{definition}
 The \emph{Borel entropy} (or just entropy) of a Borel system $(X,T)$ is:
 $$
    h(T) := \sup \{h(T,\mu):\mu\in\Pena(T)\}\in[0,\infty]\cup\{-\infty\}
 $$
\end{definition}

The \emph{topological entropy} of a continuous map $T$ on a compact metric space is defined as follows. Define the Bowen-Dinaburg balls as $B_T(x,\eps,n):=\{y\in X:\max_{0\leq k<n} d(T^kx,T^ky)<\eps\}$ and set:
 $$\begin{aligned}
    h_\topo(T)=\lim_{\eps\to0^+} h_\topo(T,\eps) &\text{ with } h_\topo(T,\eps):=\limsup_{n\to\infty} \frac1n\log r_T(\eps,n,X)\\
    & \text{ and }r_T(\eps,n,X) = \min\{\#C:\bigcup_{x\in C} B_T(x,\eps,n)\supset X\}
 \end{aligned}$$
$h_\topo(T)$ is invariant under topological conjugacy.  The \emph{variational principle} states that for $(X,T)$ is a continuous map on a compact metric space, then 
 $
    h_\topo(T) = h(T).
 $
There need not exist a measure $\mu$ such that $h(T,\mu)=h(T)$ and if it exists it need not be unique. Such measures are called \emph{measures maximizing the entropy}, or \mme\ for short.

The following theory is fundamental to the theory:

\begin{theorem}[Shannon-McMillan-Breiman]\label{t.SMBclassical}
Let $(T,\mu)$ be an ergodic measure-preserving map. Let $P$ be a countable Borel partition modulo $\mu$ such that $H(P,\mu)<\infty$. Then, for $\mu$-a.e. $x$,
 $$
   \lim_{n\to\infty} -\frac1n\log \mu(P^{T,n}(x)) = h(T,\mu,P).
 $$
\end{theorem}

\begin{xca}
Let $\alpha\in[0,1]$ and $T:[0,1]\to[0,1]$, $T(x)=2\alpha x$ for $x\leq1/2$, $T(x)=2\alpha(1-x)$ otherwise. Show $h_\topo(T)=\log\alpha$.
\end{xca}

\subsection{Empirical measures and the entropy function}

By Kuratowski Theorem \ref{thm.Kuratowski}, a standard Borel space $X$ can be assumed to be a compact metric space equipped with the Borel subsets coming from the topology. Considering the corresponding distance, for every $r>0$, $X$ can be covered by finitely many sets of diameter $<r$. Hence:

\begin{fact}\label{f.seqGenPart}
There exists a sequence of finite partitions $P_1,P_2,\dots$ such that, for all distinct $x,y\in X$, there exists $n\geq1$, $P_n(x)\ne P_n(y)$.
\end{fact}

In a topological space, one can associate to points $x\in X$ (up to an almost null set) an ergodic invariant probability measure, called the empirical measure. We shall use a Borel version of this construction:

\begin{proposition}\label{p.BorelEmpiric}
Let $(X,T)$ be a Borel system.
There is a surjective almost Borel map $M:X\to\Proberg(S)$ such that: for all Borel $P\subset\Proberg(T)$, for all $\mu\in\Proberg(T)$, 
 \begin{equation}\label{eq.charMinv}
   \mu(M^{-1}(P))>0\iff\mu(M^{-1}(P))=1\iff\mu\in P.
 \end{equation}
\end{proposition}

\begin{proof}
By the Kuratowski theorem \ref{thm.Kuratowski}, one can assume that $X$ is the Cantor set ($T$ is not necessarily continuous). In particular we can find a generating sequence  $P_1,P_2,\dots$ of finite partitions of $X$ into clopen sets such that $P_{n+1}$ is finer than $P_n\vee T^{-1}P_n$. Let $P_*$  be the countable set $\bigcup_{n\geq1} P_n$.

$[0,1]^{P_*}$ is a standard Borel space (see Prop. \ref{p.stdBorelCons}). 
We define $f:\Prob(X)\to[0,1]^{P_*}$, $\mu\mapsto(\mu(A))_{A\in P_*}$. For each $A\in P_*$, $\mu\mapsto\mu(A)$ is Borel, hence $f$ is Borel. As $P_*$ is generating, $f$ is injective so the Lusin-Souslin theorem \ref{thm.LusinSouslin} implies that $f(\Prob(X))$ is Borel and $f^{-1}:f(\Prob(X))\to\Prob(X)$ is Borel.

We define $F:X\to[0,1]^{P_*}$ by
 \begin{equation}\label{e.empFreq}
 F(x) := \left(\limsup_{n\to\infty}\frac1n\#\{0\leq k<n:S^kx\in A\}\right)_{A\in P_*}.
 \end{equation}   
It is defined for every $x\in X$, and $F:X\to[0,1]^{P_*}$ is a Borel map. 
Let
 $$
   X_1:=\{x:\liminf_{n\to\infty}\frac1n\#\{0\leq k<n:S^kx\in A\} = \limsup_{n\to\infty}\frac1n\#\{0\leq k<n:S^kx\in A\}\}
 $$
Obviously it is a Borel set and, by Birkhoff's ergodic theorem, its complement is a null set. By the choice of $P_*$ in the compact metric space $X$, the Caratheodory extension theorem yields some $\mu\in\Prob(X)$ such that $F(x)=f(\mu)$, i.e., $F(X_1)\subset f(\Prob(X))$. Hence $M:=f^{-1}\circ F$ is well-defined and Borel.

Note that for any $A\in P_*$, $T^{-1}A$ is a finite union of elements of $P_*$, hence $\mu$ is invariant. Also Birkhoff's ergodic theorem implies:
 \begin{equation}\label{eq.Mnu}
  \forall\mu\in\Proberg(S)\quad \nu(M^{-1}(\{\nu\}))=1.
 \end{equation}
The implication $\mu\in P\implies \mu(M^{-1}(P))=1$ of \eqref{eq.charMinv} follows. 

We claim that, for all $Q\subset\Prob(T)$  Borel and $\mu\in\Prob(T)$ such that, if $\mu(M^{-1}(Q))>0$, then:
 \begin{equation}\label{eq.claim}
 \exists \text{positive measure set of ergodic components $\nu$ of $\mu$ in } Q
 \end{equation} 
Indeed, the hypothesis implies that $\nu(M^{-1}(Q))>0$ for a positive measure subset of the ergodic component $\nu$ of $\mu$. Then eq. \eqref{eq.Mnu} implies $M^{-1}(\{\nu\})\cap M^{-1}(Q)\ne\emptyset$, hence $\nu\in Q$, proving \eqref{eq.claim}.

Applied to $Q=\Prob(T)\setminus\Proberg(T)$, \eqref{eq.claim} shows by contradiction that $M(x)\in\Proberg(T)$ outside a null set. Thus $M:X\setminus X_1\to\Proberg(T)$ is a well-defined, almost Borel map.

Applied to $Q=P\subset\Proberg(T)$, \eqref{eq.claim} shows that $\mu(M^{-1}(P))>0\implies \mu\in P$, concluding the proof of eq. \eqref{eq.charMinv}.

Finally, the surjectivity of $M$ follows from \eqref{eq.charMinv} with $P:=\{\mu\}$ for $\mu$ ranging over $\Proberg(T)$.
\end{proof}

\begin{xca}
Show that:
 \begin{enumerate}
   \item if $M':X\to\Prob(X)$ is another almost Borel map satisfying eq. \eqref{eq.charMinv}, then $M=M'$ except on a null set.
   \item if $(X,S)$ is a homeomorphism of a metrizable Polish space, then, for all $x$ outisde of a null set, the limit\footnote{Recall that $\delta_x$ is the probability measure such that $\delta_x(\{x\})=1$.} $\mu_x:=\lim_{n\to\infty}\frac1n\sum_{k=0}^{n-1}\delta_{S^k x}$ exists in the vague topology (generated by the continuous, compactly supported real functions on $X$) and satisfies $\mu_x=M(x)$.
    \end{enumerate}
\end{xca}

Let $\mathcal P(X)$ be the set of finite Borel partitions. 
If $A\subset X$ and $P$ is a collection of subset of $X$, $A\lessfine P$ means that $A$ is a finite union of elements of $P$. If $Q$ is a collection of subsets of $X$, then $Q\lessfine P$ means that $A\lessfine P$ for each $A\in Q$.

\begin{definition}
Let $h:\Proberg(T)\to[0,\infty]$ and, for any finite Borel partition $Q$ of $X$,  $h_Q:\Proberg(T)\to[0,\infty]$ be defined as the Kolmogorov entropies $h_T(\mu):=h(T,\mu)$ and $h_{T,Q}(\mu):=h(T,\mu,Q)$ (we often omit $T$ from the notation).
\end{definition}

\begin{proposition}\label{p.BorelEntropy}
For $(X,S)$ a Borel system and $P$ a finite Borel partition,
the functions $h,h_P:\Proberg(S)\to[0,\infty]$ are Borel. 
\end{proposition}

\begin{proof}
Let $Q_*$ be the countable set $\bigcup_{n\geq1} Q^{T,n}$ and $E:\Prob(X)\to[0,1]^{Q_*}$. Observe that $E$ is Borel as each function $\mu\mapsto\mu(A)$, $A\in Q_*$, is Borel. But  $H(Q^{T,n},\mu)$ is a continuous function of $E(\mu)$, so the following is Borel:
 $
    h_Q(\mu) = \lim_{n\to\infty}\frac1n H(\mu,Q^{T,n}).
 $
Finally, $h(\mu)$ is Borel since it is equal to $\sup_{n\geq1} h_{Q_n}(\mu)$ if $Q_n$ is a generating sequence of partitions by Sinai's theorem.
\end{proof}

\subsection{Shifts}\label{sec.back.shift}
We refer to \cite{Kitchens,GurSav} for background.
An \emph{alphabet} $\mathcal A$ is a countable (possibly finite) set with the discrete topology. Its elements are called \emph{symbols}. The \emph{full shift} on $\mathcal A$ is $(\Sigma_{\mathcal A},\sigma)$ where $\Sigma_{\mathcal A}:=\mathcal A^\ZZ$ with the product topology and the homeomorphism $\sigma:\Sigma_{\mathcal A}\to\Sigma_{\mathcal A}$ defined by $\sigma((A_n)_{n\in\ZZ})=(A_{n+1})_{n\in\ZZ}$. The \emph{cylinders} in $\Sigma_{\mathcal A}$ are the closed-open subsets: $[a_n\dots a_{n+m}]_X:=\{A\in\Sigma_{\mathcal A^\ZZ}:\forall k=n,\dots,n+m\; A_k=a_k\}$. 

A word of $X$ (or an $X$-word) of length $n$ is $w\in\mathcal A^n$ such that $[w]_X\ne\emptyset$.

 A \emph{subshift} is $(\Sigma,\sigma)$ where $\Sigma\subset\Sigma_{\mathcal A}$ is a compact, shift-invariant subset of $\Sigma_{\mathcal A}$ and $\sigma$ is the restriction of the previous homeomorphism. A subshift $\Sigma$ is said to be a \emph{Markov shift}  if there is a directed
 graph, i.e., a subset $\mathcal E\subset\mathcal A^2$ such that:
  $$
     A\in\Sigma \iff \forall n\in\ZZ \;(A_n,A_{n+1})\in\mathcal E.
  $$
 
A Markov shift $\Sigma$ is called \emph{irreducible} if it can be defined by a strongly connected graph, i.e., such that any two vertices can be joined by a loop. Any Markov shift $\Sigma$ is equal to a countably union of irreducible Markov shift (its \emph{components}) up to an almost null set. The \emph{period} of an irreducible Markov shift $\Sigma$ is the greatest common divisor of all periods of all periodic orbits of $\Sigma$. A Markov shift is called \emph{mixing} if it is irreducible and has period $1$. 

Let $\Sigma$ be an irreducible Markov shift. According to Gurevi\v{c} \cite{Gurevich1}, its Borel entropy satisfies:
 \begin{equation}\label{e.MSentropy}
     h(\Sigma) = \sup \{ h(X): X \text{ Markov shift defined by a finite irreducible subgraph}\}
 \end{equation}
and, if $h(\Sigma)<\infty$, then it has at most one \mme. In this case, $X$ is called \emph{positive recurrent}. 

Recall that a \emph{$p$-Bernoulli system} is a measure-preserving system isomorphic to the product of the circular permutation on $p$ elements and $(\NN^\ZZ,\sigma,\mu^\NN)$ for some probability measure $\mu$ on $\NN$ (note that $\mu$ may be carried by a single point). For $p\geq1$, $p$-Bernoulli system is called \emph{periodic-Bernoulli} and simple \emph{Bernoulli} if $p=1$.  By a theorem of Gurevi\v{c} \cite{Gurevich2}, the \mme's of a Markov shift with finite entropy are, if they exist Markovian measure. It is well-known that they are $p$-Bernoulli where $p$ coincides with the period of the Markov shift. We recall an immediate consequence of Ornstein theorem \cite{Ornstein}: any two periodic-Bernoulli are measure-preservingly isomorphic if and only if they have equal entropy and equal period.

Moreover, for each $t\geq0$ and $p\geq1$, there are irreducible Markov shifts $\Sigma^0_{t,p}$, $\Sigma^+_{t,p}$, which have Borel entropy $t$, period $p$ with respectively zero and one \mme.

 A \emph{shift of finite type} (or SFT) is a subshift which  can be written as $\Sigma_{\mathcal A}\setminus\bigcup_{w\in F,k\in\ZZ} \sigma^{-k}[w]$ for some finite alphabet $\mathcal A$ and finite set of words $F$. It is a \emph{one-step} SFT if it is a Markov shift defined by a finite graph. SFTs are also characterized as those subshifts topologically conjugate to one-step SFTs (see \cite{LindMarcus} for background).

If $\Sigma$ is an irreducible SFT with period $p$, given any two symbols $\alpha,\omega$, there is an integer $n_0$ and a number $C>1$ such that the number $N(n)$ of $X$-words of length $n$ whose first symbol is $\alpha$ and last symbol is $\omega$ satisfies:
 $$
   \forall n\geq n_0\quad C^{-1} \leq N(n)e^{-nh(X)} \leq C.
 $$


\subsection{Hyperbolicity}\label{sec.back.hyp}

See \cite{KH} and especially the supplement by Katok and Mendoza for background on smooth ergodic theory and Pesin theory. Let $T$ be a diffeomorphism of a compact manifold $M$. 
For $k=1,\dots,\dim M$, the $k$th \emph{Lyapunov exponent} at $x\in M$, is the following value:
 $$
  \lambda_k(T,x):= \inf_{E^{k-1}} \sup_{v\in (E^{k-1})^\perp\setminus\{0\}} \limsup_{n\to\infty}\frac1n\log\|(T^n)'.v\|_{T^nx} 
 $$ 
where $E^{k-1}$ ranges over the $(k-1)$-dimensional subspaces of the tangent space $T_xM$ and $\|\cdot\|_x$, $x\in M$, is an arbitrary Riemmanian structure and $(E^k)^\perp:=\{v\in T_xM:\forall w\in E^k\;v\cdot w=0\}$. Obviously $\lambda_k(T,\cdot)$ is a Borel function.

An ergodic measure $\mu$ of $T$ is said to be \emph{(Pesin) hyperbolic} if, for $\mu$-a.e. point there is no zero Lyapunov exponent. Observe that $\{x\in M:M(x)$ is hyperbolic$\}$ is Borel subset of $X$.

\emph{Ruelle's inequality} bounds the entropy $h(T,\mu)$ by a sum of the positive Lyapunov exponents. If $T$ is a surface diffeomorphism, this inequality applied to $(T,\mu)$ and $(T^{-1},\mu)$ shows that ergodic measures with positive entropy are hyperbolic.

A \emph{horseshoe} is an invariant subset of $M$ which is a continuous embedding of an irreducible SFT with positive entropy. Moreover, it only supports hyperbolic measures.

\section{Almost Borel Embedding and Isomorphism}\label{sec.almostBorel}

We recall basic facts about almost Borel embedding, isomorphism and universality and then compare almost Borel isomorphism to related notions.

\subsection{Universality, Embedding and Isomorphism}\label{sec.aB.Univ}
A Borel system $(X,S)$ is \emph{almost Borel universal}\footnote{Almost Borel universal systems are \lq universal terminal objects\rq\ in an obvious category.}  for a class $\mathcal C$ of Borel systems, if, for every $(Y,T)\in\mathcal C$, there is an almost Borel embedding of $(Y,T)$ into $(X,S)$. $(X,S)$ is \emph{almost Borel strictly universal} for $\mathcal C$, if it is both universal and an element of $\mathcal C$.

All classes admit almost Borel universal systems. Indeed, B. Weiss \cite{Weiss2} has shown that $(\NN^\ZZ,\sigma)$, the full-shift over a countable alphabet, is universal with respect to any class: any Borel system has a Borel embedding into it. However, not every class admits  a strictly universal system. Trivial counter-examples are the class of uniquely ergodic Borel systems or that of systems with finite entropy. Serafin \cite{Serafin} has shown that the class of selfhomeomorphisms of compact metric spaces with zero entropy has no strictly universal system.

Now, let $\mathcal B(t)$ be the class of Borel $(X,S)$ systems such that $h(S,\mu)<t$ for all $\mu\in\Prob(S)$. Hochman's Theorem \ref{thm.Hochman1} says that any Markov shift $\Sigma_t$ with entropy $t$ is universal for $\mathcal B(t)$. This gives a strictly universal system for that class as the reader is invited to check:

\begin{xca}
For $0\leq s\leq t$, let $\Sigma_t^s:=h_{\Sigma_t}^{-1}([0,s[)$. Prove that it is a Borel subsystem carrying exactly the ergodic measures of $\Sigma_t$ with entropy $<s$. Check that  $\Sigma_t^s$ is strictly universal for $\mathcal B(s)$. In particular, if $\Sigma_t$ is non positive recurrent, then it is itself strictly universal for $\mathcal B(t)$.
\end{xca}

Recall the Cantor-Bernstein theorem of set theory: if two sets embeds one into another, then they are in bijection. There is a well-known  Borel version of this theorem (see \cite[(15.7)]{Kechris}). 
As observed by Hochman, there is an equivariant version of this theorem:

\begin{lemma}[Dynamical Cantor-Bernstein \cite{Hochman}]\label{l.DCB}
Let $(X,S)$ and $(Y,T)$ be Borel systems. Assume that there are almost Borel embeddings $f:(X,S)\to(Y,T)$ and $g:(Y,T)\to(X,X)$. Then there is an almost Borel isomorphism $h:(X,S)\to(Y,T)$.
\end{lemma}

\begin{proof}
The domains of $f$ and $g$ can be assumed to be $S$- and $T$- invariant by removing almost null sets.
We further remove $X_*:=\bigcup_{n\geq0} (gf)^{-n}(X'\cup f^{-1}(Y'))$ and $Y_*$, defined symmetrically. We leave it to the reader to check these are invariant, almost null subsets of $(X,S)$ and $(Y,T)$ and that $f(X\setminus X_*)\subset Y\setminus Y_*$ and $g(Y\setminus Y_*)\subset X\setminus X_*$. Hence
we can assume that we have mutual Borel embeddings of Borel systems.

Define inductively two non-increasing sequences of subsets: $X_0=X$ and $X_{n+1}=gf(X_n)$; $Y_0=Y$ and $Y_{n+1}=fg(Y_n)$. Let
 $$
  A=\biggl(\bigcap_{n\geq0}X_n\biggr)\cup\biggl(\bigcup_{n\geq0}(X_n\setminus g(Y_n))\biggr) \text{ and }B=\bigcup_{n\geq0}(Y_n\setminus f(X_n)).
 $$
Let us assume that (*) $A$ and $B$ are Borel and that $f(A)=Y\setminus B$ and $g(B)=X\setminus A$. 

We define $h:X\to Y$ by setting $h(x)=f(x)$ if $x\in A$ and $h(x)=g^{-1}(x)$ otherwise. The claim (*) implies that it is Borel, bijective, and therefore a Borel isomorphism between the spaces using the Lusin-Souslin theorem. Finally observe that $S(X)=X$ and $T(Y)=Y$ imply that $A$ and $B$ are $S$ or $T$-invariant. Hence, using the invariance of $A$ we get: $h(S(x))=f(S(x))=T(f(x))=T(h(x))$ for $x\in A$ and likewise for $x\in X\setminus A$. The following exercice suffices to conclude.
\end{proof}

\begin{xca}\label{exo.claimCB}
Prove the claim (*) above. {\it Hint:} To analyze $A$ and $B$ one can identify, e.g., $X_n\setminus g(Y_n)$ as the set of points in $x$ such that $$x,g^{-1}x,f^{-1}g^{-1}x,\dots,(f^{-1}g^{-1})^n(x)$$ is well-defined, but not $g^{-1}\circ (f^{-1}g^{-1})^n(x)$.
\end{xca}

\begin{corollary}\label{c.uniqStUn}
Given any class of Borel systems, its almost Borel strictly universal systems are pairwise almost Borel isomorphic (if they exist). In particular, any almost Borel strictly universal system for $\mathcal B(t)$ is almost Borel isomorphic to any mixing Markov shift which has entropy $t$ and no \mme
\end{corollary}

The following lemma of \cite{Hochman} is easy but important: 

\begin{lemma}[Hochman]\label{l.SeqUnivSub}
Let $(Y,T)$ be a Borel system.
Let $H$ be the set of numbers $0\leq h\leq \infty$ such that $(Y,T)$ is $\mathcal B(H)$-universal. Then $H=[0,\sup H]$.
\end{lemma}

In particular, if for every $h<h(T)$, one can embed a mixing SFT into $(Y,T)$, then $(Y,T)$ is $h(T)$-universal. This allows Hochman \cite[Thm 1.6]{Hochman} to analyze many systems (e.g., mixing Markov shifts or natural extensions of $\beta$-shifts). This will allow us to use Katok's theorem.

\begin{proof}
As any system is $\mathcal B(0)$-universal, we can assume $\sup H>0$. For $0\leq s<t$, $\mathcal B(t)$-universality implies $\mathcal B(s)$, hence there is a sequence $h_1:=0<h_2< h_3<\dots$ converging to $\sup H$ such that $(Y,T)$ is $\mathcal B(h_n)$-universal for each $n\geq1$. Let $(X,S)$ be in $\mathcal B(\sup H)$.
According to Propositions \ref{p.BorelEmpiric} and \ref{p.BorelEntropy}, the following invariant sets are Borel:
 $$
   X_n:=\{x\in X: h_n\leq h(S,M_x) <h_{n+1}\}, 
      \quad n\geq1.
 $$
 Observe that $\bigsqcup_{n\geq1} X_n$ is a disjoint union and that it is equal to $X$ up to an almost null set. $(Y,T)$ being $\mathcal B(h_{n+1})$-universal, there is an almost Borel embedding $\psi_n:X_n\to Y$ for each $n\geq1$.
Hence we have an almost Borel homomorphism $\Psi:X\to\Sigma$ defined by $\Psi|X_n=\psi_n$ for each $n\geq1$.

For any $\mu\in\pen(\Sigma)$, if $\mu(\Psi(X_n)\cap\Psi(X_m))>0$,  $h(\sigma,\mu)\in[h_{n-1},h_{n}[\;\cap\; [h_{m-1},h_m[$ so $m=n$. Hence, $\Psi:X\to\Sigma$ is an almost Borel embedding and $(Y,T)$ is $\mathcal B(\sup H)$-universal. 
\end{proof}

The following is amusing and useful:

\begin{xca}\label{exo.countableunion}
Let  $(X,S)$, be an almost Borel strictly universal system for $\mathcal B(t)$ for some $t\geq0$. Let $\emptyset\ne I\subset\RR$. Show that $(X\times I,S\times\id)$ is almost Borel isomorphic to $(X,S)$.
\end{xca}

\subsection{Other partial Borel isomorphisms}\label{sec.partial}
We compare almost Borel isomorphism with related notions among Borel systems: Borel isomorphism, Borel isomorphism up to wandering sets and entropy-conjugacy.

It is a nice exercise to put all these notions in the following common framework of "partial Borel isomorphisms". Indeed, each of those notions corresponds to a choice of \emph{negligible subsets} in each Borel system (possibly restricted to some subclass). Then two systems are said to be isomorphic if they each contain Borel isomorphic subsystems whose complement sets are negligible in the chosen sense.

The following admissibility conditions ensure that such notions are equivalence relation (exercise):
 \begin{enumerate}
 \item[(A0)] the empty set is negligible;
 \item[(A1)] each negligible subset is contained in an invariant negligible subset;
 \item[(A2)] a finite union of negligible subsets of one system is again negligible;
 \item[(A3)] a Borel embedding of an invariant negligible subset is again neligible. 
 \end{enumerate}

\subsubsection*{Neglecting wandering sets}
Shelah and B. Weiss \cite{Shelah} (see also \cite{Weiss1,Weiss2}) have introduced and studied the following notion. A Borel set is $W$-negligible if it contained in a countable  union of wandering sets, i.e., Borel sets $W$ that are disjoint from all their iterates $T^{-n}W$, $n\in\ZZ$.  This definition was motivated by the proof of Poincar\'e's recurrence theorem in ergodic theory. A further indication of its naturalness is:

\begin{theorem}[Shelah-B. Weiss \cite{Shelah}]
A Borel subset $E\subset X$ is $W$-negligible if and only if it has zero measure with respect to all Borel probability measures $\mu$ such that $\mu$ and $\mu\circ T$ are equivalent (i.e., have the same sets of zero measure).
\end{theorem}

This notion is obviously stronger than almost Borel isomorphism. Indeed, according to B. Weiss \cite{Weiss1}: \emph{\lq the true complexity of Borel automorphisms lie in those that have no invariant measure\rq}. He offered the following question \cite[p. 397]{Weiss1}. For $\alpha\in\RR\setminus\mathbb Q$, let $R_\alpha:[0,1[\to[0,1[$ be defined by $R_\alpha(x)=x+\alpha\mod 1$. Fix $F$ a closed subset of $[0,1[$ with empty interior and positive Lebesgue measure and let $I_\alpha:=[0,1[\setminus\bigcup_{n\in\ZZ} R_\alpha^n(F)$. $I_\alpha$ has zero Lebesgue measure and is residual. Are all $R_\alpha|I_\alpha$ Borel isomorphic up to $W$-negligible sets? This question is still open.

In the classification of Markov shifts, the problem of strengthening the isomorphism is linked to the relations between ergodic and symbolic  classifications. 
The Ruelle-Perron-Frobenius theorem of Gurevi\v{c} (generalized by Sarig) define very natural classes of Markov shifts (see \cite{Kitchens}). These classes are invariant under topological conjugacy in the locally compact case and, more generally, under symbolic notions of isomorphisms like the almost isomorphism of \cite{BBG}. Hochman's theorem implies that some of these distinctions are not invariant under almost Borel isomorphism (strongly positive recurrent among positive recurrent, or null recurrent vs. transient). Would this still be the case for Borel isomorphisms? Borel isomorphism up to $W$-negligible sets? 

\subsubsection*{Entropy-conjugacy}
We turn to a weaker notion of isomorphism.
For many systems with non-uniform hyperbolicity properties, one is often led to distinguish "more tractable" measures, e.g., those that have better hyperbolicity properties. Indeed, some natural constructions will only work for those "better measure". 

One can therefore focus on invariant probability measures with nonzero entropy by taking as negligible sets those that have zero measure with respect to all ergodic, invariant probability measures with nonzero entropy. This obviously satisfies (A0)-(A3). 

A first example can be found in Hofbauer's analysis of piecewise monotone maps\footnote{More precisely, their natural extensions see...} of the interval, e.g., $C^1$ maps of $[0,1]$ into itself with finitely many critical points. Hofbauer \cite{Hofbauer} built a partial Borel isomorphism (following prior work of Takahashi \cite{Takahashi}) and showed that the discarded set had zero measure for all ergodic invariant probability measures with nonzero entropy. 

Hofbauer then used this isomorphism to analyze entropy maximizing measures and showed that whenever the topological entropy of the interval map is nonzero, then there are only finitely many ergodic entropy maximizing measures (and exactly one for maps with  a single critical point).

The author generalized such constructions to other settings: $C^\infty$ interval maps with infinite critical set \cite{BuzziSIM}, piecewise expanding and affine maps \cite{BuzziAffineMod,BuzziPSPUM} and a class of smooth maps with critical hypersurfaces \cite{BuzziBSMF,BuzziICMP}. In these settings, one is led to focus on measures with large entropy. Indeed, for some of these examples and results, some measures with positive (but small) entropy do not have the "right" properties (for instance their support is contained in hypersurfaces). 

Therefore, in these studies, one defines a Borel subset $E$ to be \emph{entropy-negligible} in some Borel system $(X,S)$ if there exists $h<h(S)$ such that $\mu(E)=0$ for all ergodic measure $\mu$ with $h(T,\mu)<h$. Conditions (A0)-(A3) are again easily checked\footnote{For this type of partial Borel isomorphism, condiiton (A2) cannot be strengthened from finite to countable.}.  The resulting notion of partial Borel isomorphism is called \emph{entropy-conjugacy}. It turns out that in all natural examples which can be analyzed up to entropy-conjugacy, one can \emph{then} apply Hochman's theorem and get an almost Borel isomorphism. We note that the known analysis of the \mme's definitely use this "entropy-conjugacy stage".


In these lectures, we shall be especially concerned with the example of surface diffeomorphisms. Here the powerful construction of Sarig  yields representations up to entropy-conjugacy (and finite fibers). More precisely, for each $\chi>0$, Sarig builds a representation up to a set negligible for all measures with entropy $>\chi$ (the complement set of a \emph{$\chi$-large} subset in Sarig's terminology). As announced in the introduction, we shall improve this to a partial Borel conjugacy up to a positive-entropy-negligible set under a mixing assumption (the general case is treated in \cite{ABS} using different methods).

\begin{remark}
In minimal dimensions (dimension $1$ for maps, $2$ for diffeomorphisms), Lyapunov exponents rather than entropy seem to be the main  phenomenon. 
Indeed, Bruin \cite{Bruin} has shown that, under a classical distortion condition
the natural partial conjugacy in a variant of Hofbauer's construction exactly discards measures with zero Lyapunov exponents. Similarly, the symbolic dynamics of Sarig discard only measures  with (some) zero exponents  (by all codings for $\chi>0$).
\end{remark}

%


%

%


\section{Borel version of Krieger's Embedding Theorem}\label{sec.BorelKrieger}

Hochman proves the universality of mixing Markov shifts by using a Borel version of Krieger's embedding theorem. 

\subsection{Statement of the Embedding Theorem}
We will first encode the Borel system by the following type of concatenations of words. 

\begin{definition}\label{def.goodSFT}
For any positive integers $s,T,N$, we consider the following sets of symbols and words:
 \begin{itemize}
  \item  $\mathcal A(s):=\{1,2,\dots,s,\sepa,\spa\}$;
  \item $\mathcal S(s,T):=\{\spa^{T-1}\sigma:\sigma=1,\dots,s\}$ and  $\mathcal T(s,T):=\{\spa^{T-1}\sepa\spa^r:0\leq r<T\}$;
  \item $\mathcal W(s,T):=\bigcup_{q\geq0} \mathcal W_q(s,T)$ with $\mathcal W_q(s,T):=\{w_1\dots w_q:w_i\in\mathcal S(s,Y)\}$;
  \item $\widehat{\mathcal W}(s,T,N):= \{tw:t\in\mathcal T(s,T)$ and $w\in\mathcal W(s,T)$ with $|tw|\geq N\}$; 
\end{itemize}
as well as the following invariant sets of sequences:
\begin{enumerate}
\item  $\Sigma_*(s,T,N)$ as the infinite concatenations of words from $\widehat {\mathcal W}(s,T,N)$; 
\item $\Sigma(s,T,N)$ as the infinite concatenations of words from $\mathcal S(s,T)\cup\mathcal T(s,T)$ such that two symbols $\sepa$ are always at least $N$ positions apart.
\end{enumerate}
We will often omit $(s,T,N)$ when they are obvious from the context.
\end{definition}


The core technical result in this section is:

\begin{theorem}\label{t.BorelKrieger}
Let $(X,S)$ be a Borel system. For any integers $s,T,N_*$ such that $\log s/T>h(X)$, there are two Borel maps: $\phi:X\to \{0,1\}^\NN$ and $\psi:X\to\Sigma_{s,T,N_*}$ such that:
 \begin{itemize}
 \item $\phi\circ S=\phi$ and $\psi\circ T=\sigma\circ\psi$;
 \item $(\phi\times\psi):X\to \{0,1\}^\NN\times\Sigma_{s,T,N}$ is injective.
 \end{itemize}
Moreover,  for any $\mu\in\pen(S)$, $\mu(\psi^{-1}([\sepa]))>0$ and the map $M:X\to\Prob(S)$ factorizes through $\phi$: $M=\tilde M\circ\phi$ for some Borel map $\tilde M:\Sigma_2\to\Prob(S)$. 
\end{theorem}

The following proof builds on the proof of Krieger's theorem presented in \cite[Thm. 4.2.3]{Downarowicz}.

\subsection{Coding $\phi$ of the measures}

The first step in the proof of Theorem \ref{t.BorelKrieger} is the following consequence of Proposition \ref{p.BorelEmpiric}:

\begin{lemma}\label{p.MeasureCoding}
Let $(X,S)$ be a Borel system. There exists a Borel map $\phi:X\to\{0,1\}^\NN$ such that, if $M:X\to\Proberg(X)$ is the almost Borel map from Proposition \ref{p.BorelEmpiric}:
 \begin{enumerate}
  \item $\phi\circ S=\phi$;
  \item $\phi=\tilde \phi\circ M$ for some Borel injective map $\tilde M:\Prob(X)\to\{0,1\}^\NN$.
 \end{enumerate}
\end{lemma}

\begin{proof} 
As $\Prob(X)\sqcup \{0,1\}^\NN$ and $\{0,1\}^\NN$ are uncountable, standard Borel spaces, the Kuratowski theorem gives a Borel injection $\tilde\phi:\Prob(X)\to\{0,1\}^\NN$. It suffices to set $\phi:=\tilde\phi\circ M$.
\end{proof}

\subsection{Basic tools for Equivariant coding}

The starting point is the Shannon-McMillan-Breiman theorem \ref{t.SMBclassical}. We rephrase it in a Borel way:

\begin{theorem}[Shannon-McMillan-Breiman]\label{t.SMBborel}
Let $(X,S)$ be a Borel system with a finite Borel partition $P$. For each $x\in X$ (up to a null set), let $h_P(x)=h(S,M_x,P)$ be the entropy of the empirical measure with respect to the partition $P$. Let
 $$
    G_P(\eps,N) := \{x\in X: \forall n\geq N\; M_x(P^n(x))=e^{-(h_P(x)\pm\eps)n}\}.
 $$
Then $h_P$ and $G_P$ are Borel and, for all $\mu\in\Proberg(S)$,
 $$
   \forall \eps>0\; \lim_{N\to\infty} \mu(G_P(\eps,N))=1.
 $$
\end{theorem}

\begin{xca}
Check that the above theorem is implied by the classical version, Theorem \ref{t.SMBclassical}.
\end{xca}

The following deduces from the above a conditional coding for some good orbit segments. One can first consider the case $Q=\{X\}$.

\begin{corollary}\label{c.SMBcoding}
Let $P,Q$ be two finite Borel partitions and let $N$ be a positive integer. Assume that $P$ is finer than $Q$. Then, for each $n\geq N$, there is a Borel function (we omit the dependence on $\eps$): $$i_{P,Q,n}:G_P(\eps,N)\cap G_Q(\eps,N)\to\NN$$ such that:
 \begin{enumerate}
  \item $i_{P,Q,n}(x)\leq e^{(h_P(x)-h_Q(x)+\eps)n}$;
  \item for all $x,y\in G_P(\eps,N)\cap G_Q(\eps,N)$ belonging to the same element of $Q^n$: $i_{P,Q,n}(x)=i_{P,Q,n}(y)\iff P^n(x)=P^n(y)$.
\end{enumerate}  
\end{corollary}

To get the equivariance of the cutting (assumed in the previous exercice), we use a (Borel version of) Rokhlin towers (see \cite[Prop. 7.9]{GlasnerWeiss}):

\begin{proposition}[Glasner-Weiss]\label{p.GWtower}
Let $(X,T)$ be a Borel system. For all $n\geq1$ and $\delta>0$, there exists a Borel set $B$ such that: (i) $B,TB,\dots,T^{n-1}B$ are pairwise disjoint; (ii) $\forall \mu\in\pen(T)$ $\mu(\bigcup_{k=0}^{p-1}T^{k}B)>1-\delta$.
\end{proposition}

We say that $B$ is the \emph{basis} of a \emph{tower} $\bigcup_{k=0}^{p-1}T^{k}B$ of \emph{height} $n$.

\newcommand\diam{\operatorname{diam}}

\begin{proof}
We can assume that there is no periodic point. For any Borel subset $A$, let $\tau_A:X\to\NN^*$ be defined by $\tau_A(x):=\inf\{n\geq1:T^nx\in A\}$. Let  $\mathcal O(A):=\bigcup_{p\in\ZZ} T^pA$.

\medbreak

\noindent{\sc Claim.} {\it  For any $n\geq 1$, there is a Borel $B\subset\mathcal O(A)$ such that $B,TB,T^2B,\dots, T^{n-1}B$ are pairwise disjoint and for all $\mu\in\Proberg(T)$, $\mu( \mathcal O(A)\setminus\bigcup_{j=0}^{n-1}T^jB)\leq n/\min_{x\in A} \tau_A(x)$.}

\begin{proof}[Proof of the Claim]
Set $B:=\{x\in X:\tau_A(x)\in n\NN^*\}$. For all $0<j<n$, $\tau_A(T^jB)\subset n\NN^*-j$ so $B\cap T^jB=\emptyset$. Also
 $
   \mathcal O(A) \setminus \bigcup_{j=0}^{n-1}T^jB \subset \bigcup_{j=0}^{n-1} T^jA.
 $
and the frequency of visits to this set is at most $n/\min\tau_A$. The claim follows from the Birkhoff ergodic theorem.
\end{proof}

Fix a Polish distance on $X$. Let $P_1,P_2,\dots$ be finer and finer partitions of $X$ with $\max_{A\in P_k}\diam(A)\to 0$. Let $N:=[n/\delta]+1$. Let $A_1,A_2,\dots$ be an enumeration of the elements of the partitions $P_k$, $k\geq1$, such that $T^jA\cap A=\emptyset$ for all $0<j<N$. Set $X_0:=X$. We define inductively subsets $B_k,X_k$ by applying the above Claim to $A_k\cap X_{k-1}$  to  to get $B_k\subset X_{k-1}$ and we set $X_k:=X_{k-1}\setminus \mathcal O(B_k)$. Let:
 $
     B := \bigcup_{k\geq1} B_k
 $
It is Borel.

We claim that $\{\mathcal O(B_k):k\geq1\}$ is a partition of $X$. The disjointness is obvious. If there were some $x\in X\setminus\bigcup_{k\geq1} \mathcal O(B_k)$, it would belong to all $X_{k-1}$, $k\geq1$. But $x,\dots,T^{N-1}(x)$ are pairwise distinct, so there must be $k\geq1$ with $x\in A_k\cap X_{k-1}$: $x\in\mathcal O(B_k)$, a contradiction. 

Finally,
 \begin{itemize}
  \item $B\cap T^jB=\emptyset$ for all $0<j<n$ is obvious from the construction. 
  \item For any $\mu\in\pen(T)$, there is, by ergodicity, some $k\geq1$ such that $\mu(\mathcal O(B_k))=1$. Hence, $\mu(X\setminus \bigcup_{j=0}^{n-1} T^jB)= \mu( \mathcal O(B_k)\setminus\bigcup_{j=0}^{n-1}T^jB_k)$, which is at most $n/N<\delta$.
 \end{itemize}
\end{proof}

\medbreak\noindent{\bf Assumptions.} 
$s,T,N_*$ are positive integers as in Theorem \ref{t.BorelKrieger}. 
Let $\eps_k$, $k\geq1$, are positive numbers with \begin{equation}\label{e.epsk}
   \sum_{k\geq1} \eps_k< \frac{1}{4}{\delta h(S)} \text{ where }
     \delta h:=\log s/T-h(S)>0.
 \end{equation}

\subsection{First level coding}
We define the first level of coding  $\phi_1$ into $\Sigma_*:=\Sigma_{s,T,N_*}$ such that $\psi_1(x)$ determine the $P_1$-name of almost all points with enough space and flexibility left for the coding of the remaining partitions. We define:
 \begin{equation}\label{eq.g1}
g_1(x):=\frac{h_{P_1}(x)+\eps_1}{\log s/T}<1 \qquad (h_{P_1}(x) := h(S,M_x,P_1)). 
 \end{equation}

For convenience, 
$[a,b[$ also denotes the integer interval $\{a,a+1,\dots,b-1\}$. $[x]$  denotes the integer part of a real number $x$ and we set: $E_T(x):=T[x/T]$.

\begin{definition}\label{def.level1}
For $\alpha\in\Sigma_*$. Define
 $$
    S_1(\alpha):=\{k\in\ZZ:\alpha_{k+T-1}=\sepa\}.
 $$ 
If $n<m$ are two consecutive elements of $S_1(\alpha)$, $[n,m[$ is called a \emph{level $1$-interval}. Let $\ell:=m-n$ with $r,q$ integers and $0\leq r<T$, $q\geq0$. Set $\ell=r+qT$ and $\bar n:=n+T+r$. Then $[n,m[$ is divided into the following integer intervals:
 \begin{itemize}
   \item $[n,\bar n]$: the \emph{marker} positions;
   \item $[\bar n,\bar n+E_T(g_1(x)\ell)-3T[$: the \emph{level $1$-filling} positions;
   \item $[\bar n+E_T(g_1(x)\ell)-3T,\bar n+E_T(g_1(x)\ell)-2T[$: the \emph{level $1$-special} positions;
   \item the remaining: the \emph{level $1$-free} positions.
 \end{itemize}
A \emph{level $1$-modification} of $\alpha$ is a sequence $\tilde\alpha\in\mathcal A^\ZZ$ such that, for each level $1$-interval, its restriction to (i) level $1$-filling positions coincides with $\alpha$; (ii) level $1$-special positions is either $\spa^{T-1}1$ or $\spa^{T-1}2$; (iii) level $1$-free positions is a word from $\mathcal W(s,T)$.

We often write $1$-interval, $1$-filling for level $1$-interval, level $1$-filling, etc.
\end{definition}

\begin{lemma}\label{f.lengths}
The following holds for $k=1$, for each $k$-interval of length $\ell$:
 \begin{enumerate}
 \item  the numbers of $k$-filling and $k$-free positions belong to $T\ZZ$. 
 \item there is at least $\ell g_k(x)-4T$ filling positions, exactly $T$ special positions and $([\ell/T]-[g_1(x)\ell/T]+1)T>(1-g_1(x))\ell$ free positions.
 \end{enumerate}
\end{lemma}

\begin{proof}
(1) is clear. To check (2), remark that the $1$-free positions in a given $1$-interval is the complement of marker, filling and special positions. Hence their number is:
 $$\begin{aligned}
   \ell -& (T+r+E_T(\ell g_1(x))-3T+T)  = ([\ell/T]T+r)-(r+E_T(\ell g_1(x))+T)\\
    &= ([\ell/T]-[\ell g_1(x)/T]+1)T > (\ell/T-[\ell g_1(x)/T])T \geq (1-g_1(x))\ell.
 \end{aligned}$$
\end{proof}

Recall that a Borel set $B$ to be \emph{almost completely positive} if for any $\mu\in\Pena(S)$, $\mu(B)>0$.

\begin{proposition}[Level $1$-coding map]\label{p.FirstCoding}
There exist a completely positive Bo\-rel subset $B_1\subset X$ and an almost Borel homomorphism $\psi_1$ defined on $X$ and into $\Sigma_*$ with the following properties for almost all $x\in X$:
 \begin{enumerate}
   \item\label{i.Bstar} for all $p\in\ZZ$, $S^p x\in B_1\iff p\in S_1(\psi_1(x))$;
   \item\label{i.inj} the $P_1,\ZZ$-name of $x$ depends only on symbols in $1$-filling positions and $M_x$;
   \item\label{i.free} the $1$-free positions occupy a fraction $>(1-g_1(x))$ of the interval;
   \item\label{i.last} the $1$-special positions and the $1$-free positions are  repetitions of $\spa^{T-1}1$;
   \item\label{i.change1} $\alpha$ as well as any of its $1$-modifications is an element of $\Sigma_*$.
 \end{enumerate}
\end{proposition} 

We first build the subset $B_1$. We abbreviate $G_{P_k}(\eps_k/3,n)$ to $G_k(n)$ for $k,n\geq1$.

\begin{lemma}\label{l.firstCuts}
For any invariant almost Borel function $N_0:X\to \NN$, there exists a Borel $B_1\subset X$ such that for almost all $x\in X$ : (i) $M_x(B_1)>0$; (ii) for any $n\geq1$, $x,S^n x\in B_1$ implies $n\geq N_0(x)$ and $x\in G_1(n)$.
\end{lemma}

\begin{proof}
For $x\in X$, we define $N_1(x)$ to be the smallest integer $n\geq N_0(x)$ such that $M_x(G_1(n))>9/10$. Theorem \ref{t.SMBborel} ensures that this is well-defined for almost all $x\in X$. $x\mapsto M_x(G_1(n))$ is an almost Borel function given $n$, so is $N_1$. $N_1(x)$ depends on $x$ only through $M_x$, so it is invariant. We can assume that it is a constant, denoted also $N_1$, by splitting $X$ into countably many invariant Borel subspaces. A subset $B_1$ for the full space will  be obtained as the union of the  subsets $B_1$ built in each of the subspaces as follows. 

Proposition \ref{p.GWtower} with $n=N_1$ and $\delta=1/10$ gives a Borel subset $B\subset X$ such that, for almost all $x\in X$, 
 $$
 M_x(B\cup\dots\cup S^{N_1-1}B)>9/10 \text{ and }
 B\cap S^kB\ne\emptyset\implies |k|\geq N_1,
 $$  
or, equivalently:
 $$
    8/10 < M_x\left(G_1(N_1)\cap\bigcup_{q=0}^{N_1-1}S^qB \right)=
    \sum_{q=0}^{N_1-1} M_x \left(G_1(N_1)\cap S^qB \right).
 $$ 
Thus, for almost all $x\in X$, there is an integer $0\leq q<N_1$ such that 
 $$
   M_x \left(G_1(N_1)\cap S^qB \right)>8/10N_1.
 $$ 
This $q$ depends only on $M_x$, so is an almost Borel and invariant function of $x$. As before, we can assume it to be constant (maybe after splitting $X$). The set $B_1:=G_1(N_1)\cap S^qB$ has the required properties.
\end{proof}

A \emph{time of visit} of $x\in X$  to $E\subset X$ is an integer $n$ such that $S^nx\in E$.

\begin{proof}[Proof of Proposition \ref{p.FirstCoding}]
Let $N>N_*:=[6\log s/\eps_1]+1$ (so $e^{2\eps_1 N/3}>s^4$) be an integer. By definition of $G_1(N)$, for all $\ell\geq N$,
 $$
    \#\{(P_1)^\ell(x):x\in G_1(N_*)\} \leq \exp\left( \ell(h_{P_1}(x)+\eps_1/3)\right)
         < s^{\ell g_1(x)/T-4}     \leq s^{[\ell g_1(x)/T]-3}.
 $$
Recall $i_{P,Q,\ell}:G_1(N)\to\NN^*$ from Corollary \ref{c.SMBcoding}:  for all $x\in G_1(N)$, $\ell\geq N$, $i_{P_1,\{X\},\ell}(x)\leq s^{[g_1(x)\ell/T]-3}$, so we can set:
 \begin{equation}\label{e.c1n}
    c_{1,\ell}(x)=c_{[g_1(x)\ell/T]-3}\circ i_{P_1,\{X\},\ell}(x) \in\mathcal W'(s,T)
 \end{equation}
where a family of injective maps $c_p:\{1,\dots,s^p\}\to \mathcal W_p(s,T)$, $p\geq1$, has been selected, depending only on $s,T,p$.

We apply Lemma \ref{l.firstCuts} to get $B_1$ with minimum return time $>N_0(x):=\max(N_*, 6\log s/\eps_1,T/(1-g_1(x))+1)$. $B_1$ being completely positive,  almost all $x$ visit it infinitely many times in the future and in the past. Hence all of $\ZZ$ is partitioned into (finite) $1$-intervals. Finally the length of each interval is at least $N_0(x)$.

To define $\psi_1(x)$ for almost all $x$, we specify $\psi_1(x)|_a^b$ for any two consecutive times $a<b$ of visit to $B_1$. Let $\ell:=b-a=qT+r$ with $q\geq0$, $0\leq r<T$. 

We set: $\sigma=1$, $u:=[\ell/T]-[g_1(x)\ell/T]+1$, and:
 \begin{equation}\label{e.FirstBlockCoding}
 \psi_1(x)|_a^b = \underbrace{\spa^{T-1}\vert\spa^r}_{\text{marker}}
       \underbrace{c_{1,\ell}(S^ax)}_{\text{1-filling}}
       \underbrace{\spa^{T-1}\sigma}_{\text{1-special}}
       \underbrace{(\spa^{T-1}1)^u}_{\text{1-free}}
 \end{equation}
To justify the annotations (marker, $1$-filling, etc.) compare with Def. \ref{def.level1}, and observe that $c_{1,\ell}(S^ax)$ and $(\spa^{T-1}1)^u$ have the right lengths  (use Lemma \ref{f.lengths}).

It is now clear that $\psi_1(x)\in\Sigma_*$. The same applies to any level $1$-modification of $\psi_1(x)$, which can only replace $\sigma$ by $2$ and $(\spa^{T-1}1)^u$ by a word from $\mathcal W(s,T)$. Claim \eqref{i.change1} is proved.

The Borel and equivariant character of $\psi_1$ are obvious from its construction. 
Claims \eqref{i.Bstar}, \eqref{i.inj}, \eqref{i.last} follow, once one observes that the decomposition of each $1$-interval into marker, filling, special and free is determined by its endpoints.
\end{proof}

\subsection{Structure of the coding}
We have just seen how to encode orbits with respect to $P_1$. We are going to do it \emph{simultaneously} with respect to all $P_k$, $k\geq1$. Obviously we cannot do these encodings independently (and get a finite entropy process), since $h_{P_k}(S,\mu)\to h(S,\mu)>0$ as $k\to\infty$. Instead we use \emph{conditional coding}: we assume that $P_1\notfiner P_2\notfiner\dots$, and remark that we only need to specify which $(P_k,I)$-name occurs given the $(P_{k-1},I)$-name for a collection of intervals $I$ covering $\ZZ$. The number of possibilities is bounded by the ratio between the measures of $(P_k,I)$-cylinders and $(P_{k-1},I)$-cylinders. We will estimate these ratios by the Shannon-McMillan-Breiman theorem in terms of the conditional entropies of $P_k$ with respect to $P_{k-1}$ for all $k\geq1$.

Of course, the  intervals $I$ above have no reason to be uniform in $k\geq1$ (they cannot be, except in very special cases). To address this, we use nested partitions of $\ZZ$ into longer and longer intervals. This  hierarchical structure will be defined by visits to a sequence of nested, completely positive sets: $B_1\supset B_2\supset\dots$, generalizing the level $1$-coding. We turn to the details.

\medbreak

Recall that positive numbers $\eps_1,\eps_2,\dots$ have been chosen small enough, see \eqref{e.epsk}. We will use the following positive Borel functions, for $k\geq1$,
 $$
     g_k(x):= \frac{h_{P_k|P_{k-1}}(S,M_x)+\eps_k+\eps_{k-1}}{\log s/T}.
 $$ 
Taking $\eps_0:=0,P_0:=\{X\}$, this is compatible with \eqref{eq.g1}.
These numbers satisfy: 
 \begin{equation}\label{e.sgk}
  \sum_{k=1}^K g_k(x) = \frac{h_{P_K}(S,M_x)+2\sum_{k=1}^{K-1} \eps_k+\eps_K}{\log s/T}\ < 1 \quad (\forall K\geq1).
 \end{equation}

\begin{definition}[Coding Structure]\label{def.codseq}
Given positive numbers $g_1,g_2,\dots$ with $\sum_{k\geq1} g_k<1$, the coding structure of a sequence $\alpha\in\Sigma_*$ is the following sequence of refining partitions of $\ZZ$ into intervals. Level $1$-intervals and their partitions into marker, $1$-filling, $1$-special, $1$-free positions, and $S_1(\alpha)$, have been defined in Def. \ref{def.level1}. For $k\geq2$, the level $k$-structure is defined inductively:
\begin{enumerate}
\item\label{cs3}  $S_k(\alpha):=\{n\in S_{k-1}(\alpha):$ the word $\spa^{T-1}2$ appears at the level $(k-1)$-special positions inside the level $(k-1)$-interval starting at $n$\};
 \item\label{cs4} $[n,m[$ is a \emph{level $k$-interval} if $n,m$ are consecutive elements of $S_k(\alpha)$. A level $j$-interval, $j<k$, contained in $[n,m[$ is a \emph{level $j$-subinterval}. $[n,m[$ is divided into
  \begin{enumerate}
    \item filling positions of each $j$-subinterval for $j<k$ and marker positions of each $1$-subinterval;
    \item the first $E_T((m-n)g_k)-T$ level $(k-1)$-free positions, called the \emph{level $k$-filling} positions;
    \item the following $T$ level $(k-1)$-free positions, called the \emph{level $k$-special} positions;
    \item the remaining $(k-1)$-free positions in $[n,m[$, called the \emph{level $k$-free} positions.
   \end{enumerate}
   \end{enumerate}
\end{definition}

\begin{figure}
\begin{center}
$$
\cdots
{\color{red}\operatornamewithlimits{\vert}_{S_3}^{n_1}}\!a^1\!\uparrow{\color{blue}b^1_1}\operatornamewithlimits{\vert}_{S_1}^{n_2}\!a^2\!\!\times{\color{blue}b^1_2\!\uparrow}{\color{red}c_1}\operatornamewithlimits{\vert}_{S_1}^{n_3}\!a^3\!\!\!\times {\color{red}c_2\!\times}{\color{green}\cdots}{\color{blue}\operatornamewithlimits{\vert}_{S_2}^{n_4}}\!a^4\!\!\uparrow{\color{blue}b^2_1}\operatornamewithlimits{\vert}_{S_1}^{n_5}\!a^5\!\times \uparrow{\color{blue}b^2_2}\!\times{\color{blue}\operatornamewithlimits{\vert}_{S_2}^{n_6}}\!a^6\!\uparrow {\color{blue} b^3\times}{\color{green}\cdots}
{\color{red}\operatornamewithlimits{\vert}_{S_3}^{n_7}}\!a^7\!\uparrow 
{\color{blue}\cdots}
$$
\caption{The above is a segment of a coding sequence as in Def. \ref{def.codseq}.
The symbols $\uparrow,\times$ or $\vert$ stand for the blocks $\spa^{T-1}2$, $\spa^{T-1}1$, or $\spa^{T-1}\vert\spa^t$, $0\leq t\leq T-1$. The letters $a^\cdot,b_\cdot^\cdot,c_\cdot^\cdot$ stand for words from $\mathcal W_q(s,T)$, $q\geq1$.  The filling and special positions are in colors (black, blue, red) corresponding to their level (except for the symbols $|$ which belong to the $1$-filling, but are colored according to the maximum $k$ such that their positions belong to $S_k$ and the level $3$-free positions which are in green. For instance $n_4\in S_2\setminus S_3$ is the left endpoint of a level $2$-interval indexed by the word $b^2_1b^2_2$ and finishing at $n_6\in S_2$. }\label{fig.coding}
\end{center}
\end{figure}

\begin{definition}
A \emph{level $k$-modification} of $\alpha\in\Sigma_*$ is  a sequence $\tilde\alpha\in\mathcal A^\ZZ$ such that, for each $k$-interval, the restriction of $\tilde\alpha$ to (i) marker or $j$-filling positions for $j\leq k$ or $j$-special positions for $j<k$ coincides with that of $\alpha$; (ii) $k$-special positions are either $\spa^{T-1}1$ or $\spa^{T-1}2$; (iii) $k$-free positions is a word from $\mathcal W(s,T)$.
\end{definition}

A straightforward induction gives:

\begin{lemma}\label{l.modifS}
If $\beta$ is a $k$-modification of $\alpha$, then $S_k(\beta)=S_k(\alpha)$.
\end{lemma}

\begin{definition}
A \emph{synchronized $T$-block} in $[n,m[$ with respect to some $\alpha\in\Sigma_*$ is an interval $[p'-(a+1)T,p'-aT[\subset]p,p'[$ where $a\in\NN$ and $[p,p'[$ is some $1$-interval contained in $[n,m[$.
\end{definition}

\subsection{Coding Map}

\begin{proposition}[All-level coding]\label{p.kCoding}
There exist completely positive Borel sets $B_1\supset B_2\supset\dots$ and almost Borel homomorphisms $\psi_1,\psi_2,\dots:X\to\Sigma_*$ with the following properties. For almost all $x\in X$, let $g_k:=g_k(x)$ and $\alpha^k:=\psi_k(x)$ for each $k\geq1$ and consider the coding structure from Def. \ref{def.codseq}. For all $j=1,\dots,k$ and almost all $x\in X$:

 \begin{enumerate}
  \item\label{i.inductive} $\alpha^k$ is a $j$-modification of $\alpha^j$. More precisely $\alpha^k$ and $\alpha^{k-1}$ only differ at $k$-filling and $(k-1)$-special positions;
  \item\label{i.BjRec} $n\in S_j(\alpha^k)\iff S^nx\in B_j$ and  $-\inf S_j(\alpha)=\sup S_j(\alpha)=\infty$.\\
\textcolor{white}{.}\hskip-1.5cm For each $k$-interval $[n,m[$ in $\alpha^k$:
  \item\label{i.kSpecial} the $(k-1)$-special positions are occupied in $\alpha^k$ by $\spa^{T-1}2$ in the first $(k-1)$-subinterval and by $\spa^{T-1}1$ in the other $(k-1)$-subintervals;
  \item\label{i.PkName}  the restriction $\alpha$ to the $k$-filling positions in $[n,m[$ is  the word $c_{k,m-n}(S^nx)$. Moreover, this word determines the $P_k,[n,m[$-name of $x$ given its $P_{k-1},\ZZ$-name; 
  \item\label{i.kFree} the number of $k$-free positions in $[n,m[$ is greater than $(m-n)(1-g_1(x)-\dots-g_k(x))$. Moreover the set of those positions is a disjoint union of synchronized $T$-blocks, each one carrying the word $\spa^{T-1}1$;
  \item\label{i.kChange} $\alpha^k$ and any of its $k$-modification belong to $\Sigma_*$.
\end{enumerate}

\end{proposition}

This somewhat technical statement (useful for its proof by induction) will yield:

\begin{corollary}\label{c.Eembedding}
There exists an almost Borel homomorphism $\psi:X\to\Sigma_*$ with the following injectivity property. There is a null set $X_0$ such that, for all $x,y\in X\setminus X_0$, if $M_x=M_y$ and $\psi(x)=\psi(y)$, then $x=y$.
\end{corollary}

\begin{proof}[Proof of Corollary \ref{c.Eembedding}]
To define $\psi:X\to\Sigma_*$, we restrict to a full set $X_1$ on which all $\psi_k(x)$ are well-defined and belongs to $\Sigma_*$.

Let $p\in\ZZ$ and $x\in X_1$. We claim that $k\mapsto\psi_k(x)|_p$ is  constant for $k\geq k(x,p)$ for some integer $k(x,p)$. Indeed, $\psi_{k+1}(x)|_p\ne\psi_k(x)|_p$ implies that $p$ is $(k-1)$-special or a $k$-filling. But this can happen only once for a given $p$. Thus we can set $\psi(x)|_p:=\psi_{k(x,p)}(x)|_p$.

To see that $\psi(x)\in\Sigma_*$, consider any $1$-interval $[a,b[$ defined by $\psi_1(x)$. Now, $\psi(x)|[a,b[=\psi_K(x)|[a,b[$ for $K=\max_{a\leq p<b} k(x,p)<\infty$. It follows that $\psi(x)|[a,b[\in\widehat{\mathcal W}(s,T,N)$. As intervals such as $[a,b[$ form a partition of $\ZZ$, it follows that $\psi(x)\in\Sigma_*$. The map $\psi:X'\to\Sigma_*$ is well-defined.
Each $\psi_k$ being Borel and equivariant so is $\psi$.
 
To prove the last assertion, observe that $\psi(x)$ is a $k$-modification of $\psi_k(x)$. Hence $\psi(x)$ together with $M_x$ determines the $P_k,\ZZ$-name of $x$ (recall Claim \eqref{i.PkName} in Proposition \ref{p.kCoding}. As the sequence of partitions $(P_k)_{k\geq1}$ separates all the points, the announced injectivity property holds.
\end{proof}

We first build the nested subsets $B_1\supset B_2\supset B_3\supset\dots$.

\begin{lemma}\label{l.secondCuts}
There exist complete positive subsets $B_1\supset B_2\supset B_3\supset \dots$ such that for all $x\in X$ and positive integers $n,k$:
 \begin{equation}\label{eq.BG}
 x,S^nx\in B_k\implies x\in G_k(n)\cap G_{k-1}(n)\text{ and }n\geq \bar N_k(x)
\end{equation}
for any sequence of Borel integer functions $\bar N_k$, that may depend on $B_1,\dots,B_{k-1}$ {\rm(}$G_0(n):=X$ by convention{\rm)}.
\end{lemma}

\begin{proof}
We proceed by induction assuming the existence of $B_1\supset\dots\supset B_k$ satisfying \eqref{eq.BG} ($B_1$ was built in Lemma \ref{l.firstCuts}).  The construction of $B_{k+1}$ is very similar to the construction of $B_1$:

Let 
$$
 N_{k+1}(x):=\min\{n\geq \max(\bar N_{k+1}(x),N_k(x)): M_x(G_{k+1}(n))>1-\frac{M_x(B_k)}{10}\}. 
 $$
By Theorem \ref{t.SMBborel}, $N_{k+1}$ is finite for almost all $x$. It is invariant, Borel and takes countably many values, hence we can assume it to be constant (similarly as in the proof of Lemma \ref{l.firstCuts}). Proposition \ref{p.GWtower} for $n=N_{k+1}$ gives a tower with basis $B\subset X$ height $N_{k+1}$ and $M_x$-measure at least $1-M_x(B_k)/10$. It follows that there is an integer $0\leq q<N_{k+1}$ (depending only on $M_x$) such that
 $$
     M_x(B_k\cap G_{k+1}(N_{k+1})\cap S^q(B))> \frac{9}{10N_{k+1}} M_x(B_k) >0. 
 $$
As before we can assume $q$ to be constant and set:
 $
   B_{k+1}:=B_k\cap G_{k+1}(N_{k+1})\cap S^q(B).
 $
To conclude, observe that $x,S^x\in B_{k+1}$ implies that $n\geq N_{k+1}\geq N_k$ so $x\in G_{k+1}(N_{k+1})\cap G_k(N_k)\subset G_{k+1}(n)\cap G_k(n)$, proving eq.\ \eqref{eq.BG}.
\end{proof}

\begin{proof}[Proof of Proposition \ref{p.kCoding}]
We will need the following encodings generalizing $c_{1,\ell}$ from eq.\ \eqref{e.c1n}. For $k,\ell\geq1$, we define $c_{k,\ell}:X\to\mathcal W_{\ell}(s,T)$  by:
 \begin{equation}\label{e.ckn}
    c_{k,\ell}(x)=c_{[g_{k}(x)\ell/T]-1}\circ i_{P_k,P_{k-1},\ell}(x) \quad (P_0:=\{X\}).
 \end{equation}
Here $c_p$, $p\in\NN$ is the same as in \eqref{e.c1n}. In particular $c_{k,\ell}(x)$ characterizes $(P_{k})^\ell(x)$ given $(P_{k-1})^\ell(x)$ if $i_{P_k,P_{k-1},\ell}(x)<s^{[g_{k}(x)\ell/T]-1}$.

We will also need the sets $B_1\supset B_2\supset\dots$ from Lemma \eqref{l.secondCuts} with parameters:
 \begin{equation}\label{e.longBj}
   \bar N_k(x)=3T/\eps_{k}+1.
 \end{equation}

We assume by induction that, for all $1\leq j\leq k$, an almost Borel map $\psi_j:X\to\Sigma_*$ has been defined satisfying all the claims \eqref{i.inductive}-\eqref{i.kChange}  (note that these hold for $k=1$ by Proposition \ref{p.FirstCoding}). 

We build $\alpha^{k+1}$ as a $k$-modification of $\alpha^k$. We call (provisionally) *-intervals the integer intervals $[n,m[$ such that $n<m$ are consecutive times of visit of $x$ to $B_{k+1}$. As $B_j\subset B_{k+1}$, any *-interval is a union of $j$-subintervals for $1\leq j\leq k$ (see Claim \ref{i.BjRec}). Inside each *-interval $[n,m[$:
\begin{enumerate}
 \item[(M1)] we write the word $\spa^{T-1}2$ into the $k$-special positions in the first $k$-sub\-in\-ter\-val of the *-interval (and keep $\spa^{T-1}1$ in  the $k$-special positions in the other $k$-subintervals).
 \item[(M2)] we write the word $c_{k+1,m-n}(S^nx)$ of length $E_T(g_{k+1}(x)(m-n))-T$ into the first $k$-free positions in $[n,m[$.
\end{enumerate}
This is possible: for (M1), this is clear; for (M2), we use Claim \eqref{i.kFree} for level $k$: each $k$-subinterval $I_i$ contains more than $|I_i|(1-g_1(x)-\dots-g_k(x))$ $k$-free positions. Summing over those subintervals, we see that $[n,m[$ contains at least $(m-n)(1-g_1(x)-\dots-g_k(x))$ level $k$-free positions. But (M2) requires only (less than) $(m-n)g_{k+1}(x)$, which is less since $\sum_k g_k(x)<1$. We now check that this $\alpha^{k+1}$ satisfies the claims for $k+1$.

Claim  \eqref{i.inductive} obviously holds from the construction and the fact that a $k$-modifica\-tion of a $j$-modification is still a $j$-modification.

Claim \eqref{i.BjRec} holds for $j\leq k$ because of the same Claim for $k$. For $j=k+1$, let $S^nx\in B_{k+1}\subset B_k$. By the induction hypothesis, $n\in S_k(\alpha^k)$, but this is the same as $S_k(\alpha^{k+1})$ by Lemma \ref{l.modifS}. Since $B_{k+1}\subset B_k$, the $k$-interval $I$ starting at $n$ is a subinterval in a unique *-interval. By the point (M1), $\spa^{T-1}2$ appears at the $k$-special positions in $I$: $n\in S_{k+1}(\alpha^{k+1})$. The converse follows from the same point (M1). Claim \eqref{i.BjRec} holds. In particular, *-intervals coincide with $(k+1)$-intervals, see Def. \ref{def.codseq}. 

Now,  Claims \eqref{i.kSpecial} and the first half of \eqref{i.PkName} are obvious consequences of the modifications (M1)-(M2).
Set $x':=S^nx$ and $\ell:=m-n$. For the second half of \eqref{i.PkName} it suffices to show:
 \begin{equation}\label{e.wts1}
 i_{P_{k+1},P_k,\ell}(x') < s^{\ell g_{k+1}(x)/T-2} < s^{[\ell g_{k+1}(x)/T]-1}
 \end{equation}
   
By construction, $x',S^\ell x'\in B_{k+1}$ implies $x'\in G_k(\ell)\cap G_{k+1}(\ell)$. Therefore,
 $$
    \frac{M_x(P_{k+1}^\ell(x'))}{M_x(P_{k}^\ell(x'))} \geq \frac{\exp \left(-h_{k+1}(x)-\eps_{k+1}/3\right)\ell}{\exp \left(-h_{k}(x)+\eps_{k}/3)\right)\ell}
 $$
Hence, using Corollary {c.SMBcoding}:
 $$\begin{aligned}
   i_{P_{k+1},P_k,\ell}(x') &\leq \exp\left(\ell\frac{\log s}T\frac{h_{P_{k+1}|P_k}(x)+\eps_{k+1}/3+\eps_k/3}{\log s/T}\right)\\
     & \leq s^{(\ell/T) (g_{k+1}(x)-(2/3)(\eps_{k+1}+\eps_k))}\\
     & 
        < s^{\ell g_{k+1}(x)/T-2},
 \end{aligned}$$
since $\ell\geq N_{k+1}(x)>3T/\eps_{k+1}$, concluding the proof of \eqref{e.wts1} and thus of Claim \eqref{i.PkName}.

To prove \eqref{i.kFree}, observe that the number of $k+1$-free positions in $[m,n[$ (obtained by substracting the $k+1$-special and $k+1$-filling positions) is larger than:
 $$\begin{aligned}
    (m-n)(1-g_1(x)-\cdots&-g_k(x)) -  (T + E_T(g_{k+1}(x)(m-n))-T)\\
       &> (m-n)(1-g_1(x)-\dots-g_k(x) - g_{k+1}(x)).
 \end{aligned}$$
Furthermore, as we wrote words from $\mathcal S(s,T)$ into the synchronized $T$-blocks constituting the $k$-free positions, the $(k+1)$-free positions are still a union of $T$-blocks carrying the word $\spa^{T-1}1$. Thus \eqref{i.kFree} holds for $k+1$.

For Claim \eqref{i.kChange}, we observe that $\alpha^k\in\Sigma_*$ and that the successive changes to produce $\alpha^{k+1}$ only involves replacing  synchronized $T$-blocs by elements of $\mathcal S(s,T)$. But this operation leaves $\Sigma_*$ unchanged. This also applies to any $k+1$-modification. Claim \ref{i.kChange} holds for $k+1$. The induction is complete.
\end{proof}

\subsection{Proof of Theorem \ref{t.BorelKrieger}}
We collect out results to prove the Borel version of Krieger's embedding theorem.

Let $\phi:X\to\{0,1\}^\NN$ be the invariant, almost Borel map from Proposition \ref{p.MeasureCoding}. As $\log s/T>h(S)$ we can apply Proposition \ref{c.Eembedding} and get an equivariant, almost Borel map $\psi:X\to \Sigma_*$. To see that these maps satisfy the claims of Theorem \ref{t.BorelKrieger}, it is enough to remark that, for all $x,y$ outside of an almost null set, $\phi(x)=\phi(y)$ implies that $M(x)$ and $M(y)$ are well-defined and equal. Therefore, $(\phi\times\psi)(x)=(\phi\times\psi)(y)$ implies that $M(x)=M(y)$ and $\psi(x)=\psi(y)$ and thus $x=y$.

\section{Hochman's Embedding Theorem}

We deduce Hochman's Theorem \ref{thm.Hochman1} from the  Borel variant of Krieger's embedding theorem. 

\newcommand\tSs{\Sigma'_*}

\subsection{Embedding into $\tSs(s,T,N)$ with $\tfrac1T\log s>h(X)$}

We first prove the Embedding Theorem assuming a uniform entropy gap and embedding into a specific class of SFTs.

\begin{definition}
Let $\Sigma'(s,T,N)$, $\Sigma_*'(s,T,N)$,  $\mathcal W'(s,T,N)$, and $\widehat{\mathcal W}'(s,T,N)$ be defined as $\Sigma(s,T,N)$, $\Sigma_*(s,T,N)$, etc.\ in Def. \ref{def.goodSFT}, but replacing $\mathcal T(s,T)$ by
 $$
   \mathcal T'(s,T):=\{\spa^{T-2}\sigma|\spa^r:0\leq r<T\text{ and }\sigma=\spa \text{  or } 1\}.
 $$
We often write $\tSs$ instead of $\tSs(s,T,N)$. 
\end{definition}

Note that $\tSs$ is invariant but not closed. It is included in the mixing SFT $\Sigma'(s,T,N)$ with $h_\topo(\Sigma'(s,T,N))\leq h_\topo(\Sigma(s,T,N)) +\tfrac1N\log 2$.

\begin{theorem}\label{t.restricted}
Given a Borel system $(X,S)$, there exists an almost Borel embedding $\Psi:(X,S)\to\tSs(s,T,N)$ for all $(s,T,N)$ such that: $\tfrac1T\log s>h(S)$ and $N$ is large enough. 
\end{theorem}

To deduce this from Theorem \ref{t.BorelKrieger}, we will replace the invariant map $\phi:X\to\{0,1\}^\NN$ by an equivariant modification of $\psi:X\to\Sigma_*$ obtained by putting symbols $1$ just before the symbols $\sepa$ at times of visit to a carefully built Borel subset $B_*$ of the Borel set $\psi^{-1}([\emptyset\sepa])$.

\subsection{Borel construction of a set with given measure}
The above subset $B_*$ will be defined using the following:

\begin{lemma}\label{l.BorelBx}
Let $Y$ be a standard Borel space.
There exists a Borel function $F:[0,1]\times\Prob(Y)\times Y\to\{0,1\}$ such that, for all $t\in[0,1]$, $\mu\in\Prob(Y)$,
 $$
     \mu(\{y\in X:F(t,\mu,y)=1\})  = t.
 $$
\end{lemma}

\begin{remark}
Hochman uses another  idea which avoids such an explicit construction.
\end{remark}

\begin{proof}
Recall the generating sequence of finite Borel partitions $P_1,P_2,\dots$. For each $k\geq1$, let $\hat P_k$ be the set of unions of elements of $P_k$. Endow $\hat P_k$ with some total order.   For each $\mu\in\Prob(Y)$, $t\in[0,1]$, let $A^\mu_n\subset Y$, $n\geq0$, be the sequence of Borel subsets defined inductively as: $A^{\mu,t}_0=Y$ and, for all $n\geq1$, $A^{\mu,t}_n$ is the first element $A$ of $\hat P_n$ (for its chosen ordering) such that
 \begin{enumerate}
 \item $A\subset A^{\mu,t}_{n-1}$;
 \item $\mu(A)\geq t$;
 \item $\mu(A)\leq \mu(B)$ for all $B\in\hat P_n$ satisfying (1) and (2).
 \end{enumerate}
The consideration of $A=A^{\mu,t}_{n-1}$ and the finiteness of $\hat P_n$ shows that $A^{\mu,t}_n$ as above always exists. 

As $\mu\mapsto\mu(A)$ is Borel for each $A\in\hat P_n$, there is a finite Borel partition of $[0,1]\times\Prob(Y)$, on each element $E$ of which, $A^{\mu,t}_n$ is constant. Hence each $F_n:E\times Y\to\{0,1\}$, $(t,\mu,x)\mapsto 1_{A^{\mu,t}_n}(x)$, is Borel. Therefore, $F_n:[0,1]\times\Prob(Y)\times Y\to\{0,1\}$ is Borel.

Define $F:[0,1]\times\Prob(Y)\times Y\to\{0,1\}$ by $F(t,\mu,x):=\inf_{n\geq1} F_n(t,\mu,x)$. $F$ is Borel as a pointwise limit of Borel functions. Let $A^{\mu,t}:=\{x\in X:F(t,\mu,x)=1\}$.

Fix some $(t,\mu)\in[0,1]\times\Prob(Y)$. By dominated convergence, $\mu(F(t,\mu,\cdot))=\lim_{n\to\infty} \mu(F_n(t,\mu,\cdot))\geq t$. Assume by contradiction that $\mu(A^{\mu,t}_n)\geq t+\eps$ for some $\eps>0$ and all $n\geq1$. Doob's martingale convergence theorem implies that, for all large $n$, there is $B_n\in P_n$ with $\mu(B_n\cap A^{\mu,t})>\tfrac12\mu(B_n)>0$. The measure $\mu$ being atomless, $\mu(B_n)<\eps$ for all large $n$. Thus, $A^{\mu,t}_n\setminus B_n\in\hat P_n$ and $t\leq \mu(A^{\mu,t}_n\setminus B_n)\leq \mu(A^{\mu,t}_n) - \mu(B_n)/2$, contradicting the definition of $A^{\mu,t}_n$.
\end{proof}

\subsection{Proof of the restricted theorem}

Let $\psi:X\to\Sigma_*$ and $\phi:X\to\{0,1\}^\NN$ be the two almost Borel maps from Theorem \ref{t.BorelKrieger}. Let $t:\{0,1\}^\NN\to[0,1]$ be the  map defined by:
 $$
    t(x) := \sum_{n\geq0} 2\cdot 3^{-n-1}\phi(x)|_n.
 $$

\begin{xca}
Show that $t:X\to[0,1]$ is almost Borel, injective and invariant.
\end{xca}

Lemma \ref{l.BorelBx} applied to the Borel space $Y=B_1$ gives a Borel function $F_1:[0,1]\times\Prob(Y)\times Y\to\{0,1\}$. We define $F:[0,1]\times\pen(S)\times X\to\{0,1\}$ by 
 $$
   F(t,\mu,x):=\alt{ F_1(t,\mu(\cdot\cap B_1)/\mu(B_1),x) & \text{if }x\in B_1\\
                     0 & \text{otherwise,}}
 $$
so $\mu(F(t,\mu,\cdot))=\mu(B_1)\cdot t$. This is well-defined since $B_1$ is completely positive. Let $B_*:=\{x\in X:F(t(x),M(x),x)=1\}$. Note that $M_x(B_*)=t(x)M_x(B_1)$ for almost all $x\in X$ since, for any $\mu\in\pen(S)$ and $\mu$-a.e. $x$, $M(x)=\mu$ and $t(x)$ is equal to some constant $t_*$, hence $\mu(B_*)=\mu(F(t_*,\mu,\cdot))$.

Let $\phi:X\to\{0,1\}^\NN$ and $\psi:X\to\tSs$ given by Theorem \ref{t.BorelKrieger}.
Let $\psi':X\to\tSs$ be defined by:
 $$
    \psi'(x)|_p = \alt{
       1 & \text{ if } S^{-p+T-2}x\in B_*\\
       \psi(x)|_p&\text{ otherwise.}
       }
 $$
Observe that $\psi'$ is Borel  because each coordinate is. It is equivariant  $\psi$ is. 
Also, $\psi'(x)\in\tSs$ since $\psi(x)\in\Sigma_*$ and the modifications only turns symbol $\spa$ into symbol $1$ just before symbol $\sepa$.

It remains to prove the injectivity of $\psi'$. It suffices to show that $\psi'(x)$ determines $t(x)$, since $\phi\times\psi$ is injective by Theorem \ref{t.BorelKrieger}. By Birkhoff's ergodic theorem, for almost all $x\in X$, the following limits
 $$
   \lim_{n\to\infty} \frac1n\{0\leq k<n: \psi'(x)|_{k+T-2}^{k+T}=1\sepa\} \text{ and }
       \lim_{n\to\infty} \frac1n\{0\leq k<n: \psi'(x)|_{k+T-1}^{k+T}=\sepa\} 
 $$
exist and are equal to $M_x(B_*)=t(x)M_x(B_1)$ and $M_x(B_1)$.  As $M_x(B_1)>0$, this shows that $\psi'(x)$ determines $t(x)$ and hence $M_x$. Theorem \ref{t.restricted} follows.

\subsection{Embedding into a given mixing SFT}
The next step is:

\begin{lemma}\label{l.embedSigma}
Let $\Sigma$ be a mixing SFT. Let $s\geq2,T\geq1$ be integers with $\tfrac1T\log s<h(\Sigma)$. For all large enough $N$,  there is an almost Borel embedding of $\Sigma_*'(s,T,N)$ into $\Sigma$.
\end{lemma}

 The following is a variant of a standard tool of symbolic dynamics (see \cite{LindMarcus}):

\begin{lemma}\label{l.marker}
Let $X$ be a one-step mixing  SFT with $h_\topo(X)>0$. If $h<h_\topo(X)$, there exist a finite $X$-word $w$ defining $X_w:=\bigcap_{n\in\ZZ} \sigma^{-n}(X\setminus[w]_X)$ and a positive integer $M$ such that:
 \begin{enumerate}
 \item no two occurences of $w$ in $X$ can overlap, i.e., for any $0<k<|w|$,  $[w]_X\cap\sigma^{-k}[w]_X=\emptyset$;
 \item for each pair of words $u,v$ in $X_w$ of lengths $\geq |w|$, there are $M$-words $s,t$ on $X$ such that $uswtv$ is a word on $X$ and, for $0\leq k<|uswtv|$, $[uswtv]_X\cap\sigma^{-k}[w]_X\ne\emptyset \implies k=|us|$.
 \item $X_w$  is a mixing SFT;
 \item $h_\topo(X_w)>h$;
\end{enumerate}
(we call a word satisfying (1) and (2) a \emph{marker}.)
\end{lemma}

\begin{proof}
As $h_\topo(X)>0$, $X$ is not reduced to a single periodic orbit and there must exist $a=a_0\dots a_{p-1}$, $a':=a_{p'}\dots a_{p-1}$ for  $1\leq p'\leq p$, $b=b_0\dots b_{q -1}$ such that, setting , $aaba'$ is an $X$-word and $a_0,\dots,a_{p-1},b_0,\dots,b_{q-1}$ are pairwise distinct symbols. For a positive multiple $m$ of $|ba'a^2|$, let:
 $$
    w:=ba'(ba'a)^m \text{ and } \tilde w:=ba'(ba'a^2)^{m|ba'a|/|ba'a^2|}.
 $$
Note, $|w|=|\tilde w|$. We claim that for $m$ large enough, $w$ has the required properties.

We prove (1) by contradiction. We must have: $w_k=w_0=b_0$, so: $0<k\leq |w|-|ba'a|$ so $ba'a_0=ba'b_0$. As $a_0\ne b_0$, this is a contradiction. Similarly,
 $$
(*)\qquad\qquad [w]\cap\sigma^{-k}[\tilde w]=[\tilde w]\cap\sigma^{-k}[w]=[\tilde w]\cap\sigma^{-k}[\tilde w]=\emptyset \text{ for all }0<k<|w|.\qquad
 $$

\medbreak

The mixing of $X$ gives an integer $M_1$ such that for any two symbols $\alpha,\beta$ of $X$ and any $n\geq M_1$ there is an $X$-word $s$ of length $n$ such that $\alpha s\beta$ is an $X$-word. We assume that $m$ is large enough so that $|w|>M:=M_1+2|a|+|b|$.

To prove Claim (2) let $u,v$ be $X_w$-words of lengths $\geq|w|$. We consider overlaps involving $v$. Note $v\in X_w$ so $k<|uswt|$. Let $j\geq0$ be the largest integer such that
 $$
     v_0\dots v_{j-1} = w_{\ell-j}\dots w_{\ell-1}
 $$
\noindent{\sl First case:} if $j=0$, then no overlap is possible so any $v$ as above fulfills the Claim.

\noindent{\sl Second case:} if $j<|a'(ba'a)^m|$ and $v_0=b_i$, we set $c=a^2b_0\dots b_{i-1}$ ($c:=a^2$ if $i=0$) so $cv$ is an $X$-word. Pick a word $\tilde t$ such that $w\tilde tc$ is a $X$-word. The same argument as in (1) shows that $w$ cannot overlap with $tv$ for $t:=\tilde t c$.

\noindent{\sl Third case:} if $j<|a'(ba'a)^m|$ and $v_0=a_i$, this is entirely similar as above using $c=aa_0\dots a_{i-1}$.

\noindent{\sl Fourth case:} if $j\geq |a'(ba'a)^m|$, remark that $v\in X_w$, gives a word $c$ of length $|w|-j\leq |b|$ such that $cv_0\dots v_{j-1}\ne w$. One concludes as before. 

Thus, for all $v$, one can find $t$ excluding any overlap with $k>|usw|$. $|us|\leq k <[usw|$ is forbidden by Claim 1. Similarly, one can find $s$ depending on $u$ excluding any overlap with $k<|us|$. Claim (2) is proved.

\medbreak

We turn to Claim (3). $X_w$ is defined from $X$ by forbidding a single word, hence it is a (possibly multi-step) SFT. The mixing will follow if, for any two $X_w$-words $u,v$  and integer $n\geq 2M+|w|$, there is a word $t$ of length $n$ such that $utv$ is an $X_w$-word. Claim (2) gives a word $swt$ with length $n$ such that $u(swt)v$ is a $X$-word where $w$ only appear in the obvious place. Now $u(s\tilde wt)v$ is a $X$-word in which $w$ cannot appear: otherwise, by Claim (2) it would overlap $\tilde w$, but this is not possible by the property (*) above.

\medbreak

To prove Claim (4), associate to each $X$-word $u$ of length $n$ a $X_w$-word by replacing each occurence of $w$ in $u$ with a copy of $\tilde w$ and trimming the first and last $|w|$ symbols. The non-overlapping properties of $w$ and $\tilde w$ with respect to themselves and with respect to each other show that this map is at most $\#\mathcal A^{2|w|}\cdot 2^{n/|w|}$-to-$1$. It follows that, for $|w|$ large enough,
 $$
    h_\topo(X)\leq h_\topo(X_w) + \frac{\log 2}{|w|} < h_\topo(X_w) + h.
 $$
\end{proof}

\begin{proof}[Proof of Lemma \ref{l.embedSigma}]
Lemma \ref{l.marker} applied to $\Sigma$ and $h_2$ gives a marker $w$ and an integer $M$. Since $h_\topo(\Sigma_w)>h_2$, there exist numbers $C_1,\eps>0$ and $L_1<\infty$ such that, for any $\ell\geq L_1$, there exist $N(\ell)\geq C_1e^{\ell(h_2+\eps)}$ distinct $\Sigma_w$-words of length $\ell$:
 $
   u^{\ell,1},\dots,u^{\ell,N(\ell)}.
 $
By Claim (2) of Lemma \ref{l.marker}, there are pairs $t^{\ell,i}$, $s^{\ell,i}$ of $\Sigma$-words of length $M$  such that:
 $
   ws^{\ell,i}u^{\ell,i}t^{\ell,i}w$ is a $\Sigma$-word for $i=1,\dots,N(\ell).
 $ 
 There is $N_*<\infty$ such that for all $\ell\geq N_*$, the following words of length $\ell$ are pairwise distinct:
 $$
 v^{\ell,i}:=ws^{\ell-M,i}u^{\ell-M,i}t^{\ell-M,i}, \quad i=1,\dots, [e^{h_2\ell}],
 $$
as $N(\ell-2M-|w|)\geq C_1 e^{-h_2 (|w|+2M)}e^{\eps \ell}\cdot e^{h_2 \ell}$.
For $N\geq N_*$, we define the embedding $i:\tSs(s,T,N)\to\Sigma$ as follows. For  $\alpha\in\tSs(s,T,N_*)$, we define:
 $$
   \{\dots < p_{-1} < 0 \leq p_0 < p_1 <\dots\} := \{p\in\ZZ: \alpha|_{p}^{p+T} = \spa^{T-2}\sigma\sepa,\; \sigma=\spa,1\}.
  $$
For each $n\in\ZZ$, the word $\alpha|_{p_n}^{p_{n+1}}$ belongs to $\widehat{\mathcal W}'(s,T)$. The number of words of length $\ell$ in this latter set is at most: $2s^{[\ell/T]-1}\leq s^{\ell/T}$. Indeed, if $\ell=qT+r$ with $0\leq r<T$, $q\in\NN$, then the first $T+r$ symbols are $\spa^{T-2}\sigma\sepa\spa^r$ with $\sigma=\spa$ or $1$ and the remainder is a concatenation of words from $\mathcal S(s,T)$. 

For each $\ell\geq N$, we fix an enumeration 
 $$
   \{w^1,\dots,w^{n(\ell)}\} = \widehat{\mathcal W}'(s,T)\cap\mathcal A^\ell
 $$
and we define $i(\alpha):=\beta$ where $\beta|_{p_n}^{p_{n+1}}:=v^{p_{n+1}-p_n,j}$  if $\alpha|_{p_n}^{p_{n+1}}=w^j$. Obviously the map $i$ is Borel and equivariant.
To see that $i:\tSs(s,T,N)\to\Sigma$ is injective, it suffices to see that the marker $w$ occurs in $i(\alpha)$ exactly at positions $p_n$, $n\in\ZZ$. This follows from Claim (2) from Lemma \ref{l.marker}. 
\end{proof}

\subsection{General case}
We now prove Hochman's embedding Theorem \ref{thm.Hochman1}.
Recall the class $\mathcal B(h)$ of Borel systems from Sec.\ \ref{sec.aB.Univ}.

\begin{proof}[Proof of Theorem \ref{thm.Hochman1}]
By Lemma \ref{l.SeqUnivSub}, it suffices to show that, for any $h<h(\Sigma)$, any Borel system $(X,S)$ from $\mathcal B(h)$ has an almost Borel embedding into $\Sigma$.

By eq. \eqref{e.MSentropy}, there is a mixing SFT $\bar\Sigma\subset\Sigma$ with $h_\topo(\bar\Sigma)>h$. Fix integers $s\geq2$, $T\geq 1$ such that $h<\tfrac1T\log s<h_\topo(\bar\Sigma)$.  Theorem \ref{t.restricted} gives an almost Borel embedding $\psi:X\to\tSs(s,T,N)$ for all large $N$. Proposition \ref{l.embedSigma} gives an almost Borel embedding $i:\tSs(s,T,N)\to \bar\Sigma$, for all large $N$. Thus, picking one such large integer $N$, $\Psi:=i\circ\psi:X\to\Sigma$ is an almost Borel embedding.
\end{proof}

\section{Katok's theorem with periods}\label{sec.Katok}

In this section, we recall the classical theorem by Katok approximating non-trivial hyperbolic measures by horseshoes and supplement it by relating the periods of the measures and that of the horseshoes. In particular a mixing hyperbolic measure is approximated by mixing horseshoes. Combining this with Hochman's Theorem \ref{thm.Hochman1}, we will obtain Theorem \ref{t.mixdiff} and then Corollaries \ref{c.sarig}, \ref{c.berger}, and \ref{c.BF}.

\subsection{Periods of an ergodic system}\label{sec.periods}
If $\mu\in\Proberg(T)$, 
 $$
   \per(T,\mu):=\{p\geq1:\exists f\in L^2(\mu)\; f\circ T = e^{2i\pi/p}f \text{ and }|f| \equiv 1\}.
 $$
The integers $p\in\per(T,\mu)$ are called the \emph{periods} of $(T,\mu)$. Some measures have a \emph{maximum period} $p$, i.e., $p$ is a period and no larger integer is a period.

\begin{xca}
An integer $p\geq1$ is a period of $(T,\mu)$ if and only if there exists Borel subsets $X_0,\dots,X_{p-1}$ such that $TX_i=X_{i+1}$  (here $X_p:=X_0$) and $\{X_0,\dots,X_{p-1}\}$ is a partition of $X$ (both assertions modulo $\mu$). Check that $(T,\mu)$ totally ergodic implies that only $1$ is a period. Prove the converse.
\end{xca}

\begin{xca}
Check that any positive divisor of a period of $(T,\mu)$ is again a period of $(T,\mu)$.
Give examples of ergodic measure-preserving systems $(T,\mu)$, the set of periods of which coincides with (1) $\{2^n:n\in\NN\}$; (2) $\{2^p3^q:p,q\in\NN,\; q\leq 10\}$. Characterize the sets of periods of ergodic systems among the subsets of $\NN^*$.
\end{xca}

 
\subsection{Katok's theorem}

\begin{theorem}\label{t.KatokPeriod}
Let $f$ be a $C^{1+}$-diffeomorphism of a compact manifold. Let $\mu$ be an ergodic, aperiodic, hyperbolic invariant probability measure and let $\eps>0$. Then there exists a  horseshoe, i.e., a continuous embedding of an SFT $S$, such that
 \begin{enumerate}
  \item $h_\topo(S)>\max(h(f,\mu)-\eps,0)$;
  \item the period of $S$ is a period of $(f,\mu)$.
 \end{enumerate}
Moreover, if $\mu$ is totally ergodic, then the SFT $S$ is mixing.
\end{theorem}

This theorem is due to Katok (\cite{KatokIHES} for the existence of a horseshoe, \cite{KH} for the approximation in entropy), except for the new additional information (2) about the period of the horseshoe, which we now explain.

Let $(S,\mu)$ be an ergodic measure-preserving system. Let $A$ be a Borel set with $\mu(A)>0$ and let $\tau_A:X\to\{1,2,\dots,\infty\}$ be the time to $A$: $\tau_A(x):=\inf\{n\geq1:S^nx\in A\}$.

\begin{lemma}\label{l.tauPeriod}
If $\tau_A(A)\subset p\NN$, then $p$ is a period of $(S,\mu)$.

If $(S,\mu)$ has a period $q$, then any set $A$ of positive measure contains a subset $B$ of positive measure such that the greatest common divisor $\gcd(\tau_B(B))$ is a multiple of $q$ and a period of $(S,\mu)$.

In particular, if $q$ is the maximum period of $(S,\mu)$ then $q=\gcd(\tau_B(B))$.  
\end{lemma}

\begin{proof}
For $0\leq i<p$, let $A_i:=\{x\in X:\tau_A(x)\in -i+p\ZZ\}$. $\mathcal A:=\{A_0,\dots,A_{p-1}\}$ is a partition of $X$ modulo $\mu$, since $\tau_A$ is finite $\mu$-a.e. on $X$ by ergodicity.  If $0\leq i<p-1$, then $S(A_i)\subset X\setminus A$ so $\tau_A(Sx)=\tau_A(x)-1$ for all $x\in A_i$ and $S(A_i)=A_{i+1}$. Now, $S(A_{p-1})\subset A_0\cup A$, but by assumption $A\subset A_0$, so $S(A_{p-1})=A_0$. 
Thus, $\mathcal A$ is a $p$-cyclically moving partition modulo $\mu$ and $p$ is a period of $\mu$. The first claim of the Lemma is proved.

Let $q$ be a period of $(S,\mu)$, so there exists a $q$-cyclic partition modulo $\mu$, i.e., a Borel set $X_0\subset X$ such that $S^q(X_0)=X_0$ and $\{X_0,S(X_0),\dots,S^{q-1}X_0\}$ are disjoint with union of measure $1$.

For some $0\leq i<q$, $B:=A\cap S^iX_0$ has positive measure.  Obviously $\tau_{B}(B)\subset q\NN\cup\{\infty\}$. By removing points that don't return infinitely often to $B$ (a $\mu$-negligible subset), we exclude the infinite value.  
Let $p:=\gcd(\tau_B(B))$. Obviously, $p$ is a multiple of $q$ and $\tau_B(B)\subset p\NN$, so $p$ is a period by the first part of this lemma. The second claim of the Lemma is proved.

The last claim is now immediate.
\end{proof}

\begin{proof}[Proof of Claim (2) in Theorem \ref{t.KatokPeriod}]
We consider Theorem S.5.9 and its proof in \cite[pp. 698-700]{KH}. We will use the notations from this text. The horseshoe is constructed by considering a Pesin set $\Lambda_\delta$ (a non-invariant compact set with good hyperbolicity estimates) and finding a large enough set $D_m$ of well-separated points that return after $m$ iterates to the Pesin set, near to themselves, for arbitrarily large times $m$. The horseshoe is then constructed as the set of orbits that shadow arbitrary concatenations of the previously mentioned orbit segments. In this way we have an SFT $\Sigma$ defined by those concatenations and a continuous factor map $\pi:\Sigma\to X$ defined by this shadowing. According to \cite[pp. 698-700]{KH}, $\pi(\Sigma)$ is the announced horseshoe with entropy $\log\#D_n/n$. 

We concentrate on the case $h(f,\mu)>0$ as the following arguments are easily adapted to the situation where $\mu$ is only assumed to be aperiodic (and then  $\# D_n\geq2$ is enough).

The lower bound (close to $\exp n h(S,\mu)$) for $\#D_n$ is obtained from a formula for the entropy also established in \cite{KatokIHES}. This formula shows that if $\mu$ is an ergodic measure, $\delta>0$, and if $N(n,\eps,\delta)$ is the minimal number of $(\eps,n)$-dynamical balls, the union of has measure $\geq1-\delta$,
 $$
   h(S,\mu) = \lim_{\eps\to0}\liminf_{n\to\infty} \frac1n\log N(n,\eps,\delta).
 $$
We apply this not to the Pesin set $\Lambda_\delta$ but to a subset with positive measure. Lemma \ref{l.tauPeriod} shows that one can find a set of positive measure $B\subset\Lambda_\delta$ such that $p:=\gcd(\tau_B(B))$ is a period of $(f,\mu)$ (the maximum period of $\mu$ if it exists). Applying Katok's construction, we get an SFT with a large period $m$. By its construction, $B$ contains finitely many points whose return times have greatest common divisor $p$. We add them to the previously mentioned set $D_m$. The entropy of the image of the resulting SFT, $\Sigma$, can only increase.

To conclude, it is convenient to use the result on continuous factors of SFTs from \cite{ABS}:  $\Sigma$ contains another SFT $\Sigma'$ with the same period and  $h_\topo(f,\Sigma')$ arbitrarily close to $h(f,\pi(\Sigma))$ such that $\pi|\Sigma'$ is injective.
\end{proof}

\subsection{Diffeomorphisms with hyperbolicity and mixing}\label{sec.hypmix}
We prove Theorem \ref{t.mixdiff} about diffeomorphisms with hyperbolicity and mixing and then its Corollaries.

\begin{proof}[Proof of Theorem \ref{t.mixdiff}]
We assume $h_\topo(T)>0$, since otherwise point (1) holds with $M_1=\emptyset$ and $M_2=M$ and point (2) is then trivial (the periodic-Bernoulli systems of zero entropy are the finite circular permutations).

Recall from Propositions \ref{p.BorelEmpiric} and \ref{p.BorelEntropy} the Borel maps $M:X\to\Proberg(T)$ and $h:\Prob(T)\to\RR$. Hence $M_2:=(h\circ M)^{-1}(\{h_\topo(T)\})$ is Borel and carries exactly the \mme's. 

We are going to show that $M_1:=M\setminus M_2$ is strictly universal for $\mathcal B(h_\topo(T))$ (see Sec.\ \ref{sec.aB.Univ}). By Corollary \ref{c.uniqStUn}, this will imply that $(M_1,T)$ is almost Borel isomorphic to a non positive-recurrent, mixing Markov shift, say $\Sigma_0$. Obviously $(M_1,T)\in\mathcal B(h_\topo(T))$. By Lemma \ref{l.SeqUnivSub}, it suffices to show that $M_1$ is $\mathcal B(h)$-universal  for $h<h_\topo(T)$, arbitrarily close. 

For such $h$, let $\mu\in\pen(T)$ be totally ergodic and hyperbolic with $h(T,\mu)>h$ as in assumption (\#). Theorem \ref{t.KatokPeriod} gives a mixing horseshoe $H\subset M$ with $h_\topo(T|H)>h$. By Theorem \ref{thm.Hochman1}, $H$ is $h$-universal. This completes the proof of point (1).

\medbreak

We turn to point (2). First (\S) is necessary since any ergodic measure of a Markov shift is carried by one of its countably many irreducible component, that each irreducible component carries at most one \mme\ and that this measure, if it exists, is period-Bernoulli.

Conversely, we build an isomorphic Markov shift assuming (\S).  For each $p\geq1$, let $E_p$ be the (empty, finite or countably infinite) set of m.m.e.'s that are $p$-Bernoulli. Let $\Sigma_p$ be a positive recurrent Markov shift with period $p$ and entropy $h_\topo(T)$ and $\hat\Sigma_p:=\{x\in\Sigma_p:h(\sigma,M(x))=h_\topo(T)\}$. We claim that $T$ is almost Borel isomorphic to the Markov shift:
 $$
    \Sigma =  \Sigma_0 \;\sqcup\;  \bigcup_{p\geq1} \Sigma_p\times(E_p,\id).
 $$
where $\Sigma_0$ is the mixing Markov shift introduced above. Note that $\Sigma_0$ is strictly $\mathcal B(h_\topo(T))$-universal.

Indeed, $\Sigma_0$ is almost Borel isomorphic to $\Sigma_0\times(\NN,\id)$ (see Ex. \ref{exo.countableunion}) and $\Sigma_0\sqcup\Sigma_p\setminus\hat\Sigma_p$ is almost Borel isomorphic to $\Sigma_0$. Hence ($\equiv$ representing almost Borel isomorphisms),
 $$\begin{aligned}
   \Sigma &\equiv \Sigma_0\times\NN \;\sqcup\;  \bigcup_{p\geq1} \left((\Sigma_p\setminus\hat\Sigma_p)\times E_p \sqcup \hat\Sigma_p\times E_p\right)\\
      &\equiv \Sigma_0 \;\sqcup\;  \bigcup_{p\geq1} \left(\Sigma_0\times E_p\sqcup\hat\Sigma_p\times E_p\right)\\
      &\equiv \Sigma_0  \;\sqcup\;  \bigcup_{p\geq1} \hat\Sigma_p\times E_p
      \equiv M_1\sqcup M_2 \equiv T,
 \end{aligned}$$
using Ornstein theory in the step before the last.
\end{proof}

\begin{proof}[Proof of Corollary \ref{c.sarig}]
Let $T$ be a $C^{1+}$-diffeomorphism of a compact surface with $h_\topo(T)>0$ and a mixing \mme\ $\mu_*$. Ruelle's inequality implies that $\mu_*$ is hyperbolic (see Sec.\ \ref{sec.back.hyp}). Also the \mme's are periodic-Bernoulli and countably many according to Sarig \cite{Sarig2}.

Thus Theorem \ref{t.mixdiff} shows that $T$ is almost Borel isomorphic to a Markov shift. The data in Corollary \ref{c.sarig} are clearly invariant. For the converse, observe that these data determine the Markov shift $\Sigma$ built in the proof of Theorem \ref{t.mixdiff}.
\end{proof}

\begin{proof}[Proof of Corollary \ref{c.berger}]
Let $T=H_{a,b}$ be a H\'enon-like map for a good parameter $(a,b)$ in the sens of Berger \cite{Berger}. According to that work, $T$ has a unique \mme\ which is Bernoulli and hyperbolic. Restricting to a bounded, open forward invariant set, we can apply Corollary \ref{c.sarig}. We obtain an almost Borel isomorphism to a positive-recurrent, mixing Markov shift with entropy $h(T)$.  
\end{proof}

\begin{proof}[Proof of Corollary \ref{c.BF}]
 According to Theorem 1.2 of \cite{BF}, these diffeomorphisms are entropy-conjugate to the initial Anosov diffeomorphism. Hence they have a unique \mme\ and this \mme\ is Bernoulli. This \mme\ is also hyperbolic by the proof of the injectivity of the factor map denoted by $\pi$ in the end of Sec. 6.2 of \cite{BF}. One concludes as in the proof of Corollary \ref{c.berger}. 
\end{proof}

\section{Diffeomorphisms beyond the mixing case}\label{sec.notmixing}

In this section we extend the analysis of diffeomorphisms of Theorem \ref{t.mixdiff}, still relying only on Katok's and Hochman's theorems. We use the notion of an union-entropy-periodic universal part from \cite{ABS} (which we will not re-prove here).

\subsection{The universal part}
Recall from Sec. \ref{sec.periods} the set of periods $\per(S,\mu)$ of an ergodic system $(S,\mu)$.

\begin{definition}\label{d.epdom}
A measure $\nu\in\pen(S)$  \emph{entropy-period dominates} a measure $\mu\in\pen(S)$ if  (1) $\per(S,\nu)\subset\per(S,\mu)$; and (2) $h(S,\nu)>h(S,\mu)$.
\end{definition}

\begin{theorem}\label{t.domdiff}
Any  $C^{1+}$-diffeomorphism $T$ of a compact manifold $M$, contains a Borel subsystem $M_1$ which:
 \begin{itemize}
  \item carries all ergodic measures which are entropy-period dominated by some hyperbolic measure;
  \item is almost Borel isomorphic to a Markov shift.
 \end{itemize}
\end{theorem}

\begin{proof}
Recall the following from \cite{ABS}. Any Borel system such as $(M,T)$ has a subsystem $(M_U,T_U)$ called its \emph{union-entropy-period universal part} and a sequence $u_T:\NN^*\to[0,\infty]$ called its \emph{universality sequence} such that ($\Sigma_{t,p}^0$ denotes any irreducible Markov shift with entropy $t$ and period $p$ and no \mme):
 $$
   u_T(p) := \sup\{t>0: \exists \Sigma^0_{t,p} \text{ that almost Borel embeds into } (M,T)\}.
 $$
with the following properties:
 \begin{itemize}
   \item $(M_U,T_U)$ is almost Borel isomorphic to a Markov shift;
   \item it carries all $\mu\in\Pena(S)$ such that $h(T,\mu)<u_T(p)$ for some $p\in\per(T,\mu)$.
 \end{itemize}
We set $M_1:=M_U$.
Now, let $\mu\in\Pena(T)$ be entropy-period dominated by a hyperbolic ergodic measure $\nu$. Applying Katok's Theorem \ref{t.KatokPeriod}  to $(T,\nu)$ yields a continuous embedding into $(M,T)$ of some  irreducible SFT $\Sigma$ with period $p\in \per(T,\mu)$ and $h_\topo(\Sigma)> h(T,\mu)$. It follows that $h(f,\mu)<u_T(p)$. Thus $\mu(M_1)=1$.
\end{proof}

\subsection{Proof of Theorem \ref{t.hypdiff}}
The statement is in term of the \epm\ measures, which generalize  \mme's (see Def. \ref{d.epm}) (Note, that $\mu$ is \epm\ if and only if it is not dominated by any other measure).

\begin{xca}
Let $\Sigma^+_{t,p}$ be an irreducible Markov shift which has period $p$, entropy $t$ and is  positive recurrent. Determine the \epm\ measures of  $\Sigma=\Sigma^+_{1,1}\cup\Sigma^+_{2,2}\cup\Sigma^+_{2,3}\cup\Sigma^+_{6,4}$.
\end{xca}

\begin{proof}[Proof of Remark \ref{rem.epmHyp}]
This is a direct consequence of Katok's theorem: any aperiodic hyperbolic ergodic measure with zero entropy is dominated by another measure which is hyperbolic with positive entropy.
\end{proof}

We deduce Theorem \ref{t.hypdiff} from Theorem \ref{t.domdiff}. The latter yields a Borel subsystem $M_1$, almost Borel isomorphic to a Markov shift.

Propositions \ref{p.BorelEmpiric} and \ref{p.BorelEntropy} show that $M_*:=\{x\in M:h(T,M(x))<h_0\}$ is a Borel subset such that $\mu(M_*)=1$ if and only if $h(T,\mu)< h_0$.

We set $M_0:=M_*\setminus M_1$ and $M_2:=M\setminus(M_0\cup M_1)$. Obviously $M=M_0\sqcup M_1\sqcup M_2$ is an invariant Borel partition of $M$ and claims (1) and (2) are clear.

Let $\mu\in\Pena(T)$ with $\mu(M_2)=1$.  As $\mu(M_*)=0$, $h(T,\mu)\geq h_0$. If $\mu$ is dominated by some measure $\nu$, then $h(T,\nu)>h(T,\mu)\geq h_0$ hence $\nu$ is hyperbolic, contradicting $\mu(M_1)=0$. Therefore $\mu$ is not dominated by any measure: it is \epm.
Theorem \ref{t.hypdiff} is proved.

\bibliographystyle{amsplain}

\providecommand{\bysame}{\leavevmode\hbox to3em{\hrulefill}\thinspace}
\providecommand{\MR}{\relax\ifhmode\unskip\space\fi MR }
\providecommand{\MRhref}[2]{%
  \href{http://www.ams.org/mathscinet-getitem?mr=#1}{#2}
}
\providecommand{\href}[2]{#2}
\begin{thebibliography}{}

\end{thebibliography}


\begin{thebibliography}{99999}
\bibitem{AM}
R. Adler, B. Marcus, Topological entropy and equivalence of dynamical systems. Mem. Amer. Math. Soc. 20 (1979), no. 219.
\bibitem{BeguinCrovisierLeRoux}
F. B\'eguin, S. Crovisier, F. Le Roux, Construction of curious minimal uniquely ergodic homeomorphisms on manifolds: the Denjoy-Rees technique. Ann. Sci. \'Ecole Norm. Sup. (4) 40 (2007), 251--308. 
\bibitem{Berger}
P. Berger, Properties of the maximal entropy measure and geometry of H\'enon attractors, preprint 	arXiv:1202.2822 [math.DS].
\bibitem{BV}
C. Bonatti, M. Viana, S.R.B. measures for partially hyperbolic systems whose central direction is mostly contracting, Israel J. Math. 115 (2000), 157--193.
\bibitem{ABS}
M. Boyle, J. Buzzi, {\it The almost Borel structure of surface diffeomorphisms, Markov shifts and their factors,} in preparation. 
\bibitem{BBG}
M. Boyle, J. Buzzi, R. Gomez, Almost isomorphism for countable state Markov shifts. J Reine Angew Math 592 (2006) 23--47.
\bibitem{Bruin}
H. Bruin, Induced maps, Markov extensions and invariant measures in one-dimensional dynamics, Comm. Math. Phys. 168 (1995), 571--580.
\bibitem{BuzziSIM}
J. Buzzi, Intrinsic ergodicity of smooth interval maps. Israel J. Math. 100 (1997), 125--161.
\bibitem{BuzziAffineMod}
J. Buzzi, Intrinsic ergodicity of affine maps  in $[0,1]^d$. Monatsh. Math. 124 (1997), 97–-118.
\bibitem{BuzziPSPUM}
J. Buzzi, Thermodynamical formalism for piecewise invertible maps: absolutely continuous invariant measures as equilibrium states. Smooth ergodic theory and its applications (Seattle, WA, 1999), 749–-783, Proc. Sympos. Pure Math., 69, Amer. Math. Soc., Providence, RI, 2001.
\bibitem{BuzziBSMF}
J. Buzzi, Ergodicit\'e intrins\`eque de produits fibr\'es d'applications chaotiques unidimensionelles.  Bull. Soc. Math. France 126 (1998), 51--77.
\bibitem{BuzziICMP}
J. Buzzi, Dimensional entropies and semi-uniform hyperbolicity in: New Trends in Mathematical Physics: Selected contributions of the XVth International Congress on Mathematical Physics,
  V. Sidoravicius (ed.)  (also as arXiv:1102.0612 [math.DS]).
\bibitem{BF}
J. Buzzi, T. Fisher, Entropic stability of some robustly transivitve non partially hyperbolic diffeomorphisms, J. Mod. Dynamics, to appear.
\bibitem{DGS}
M. Denker, C. Grillenberger, K. Sigmund, {\it Ergodic theory on compact spaces}, Lecture Notes in Math. 527, Springer, 1976.
\bibitem{Downarowicz}
T. Downarowicz, Entropy in Dynamical Systems (New Mathematical Monographs, Vol. 18), Cambridge University Press, 2011.
\bibitem{GlasnerWeiss}
E. Glasner, B. Weiss, On the interplay between measurable and topological dynamics. Handbook of dynamical systems. Vol. 1B, 597–648, Elsevier B. V., Amsterdam, 2006.
\bibitem{Gurevich1}
B. M. Gurevi\v{c}, Topological entropy of a countable Markov chain, Dokl. Akad. Nauk SSSR 187 (1969), 715--718. 
\bibitem{Gurevich2}
B. M. Gurevi\v{c}, Shift entropy and Markov measures in the space of paths of a countable graph, Dokl. Akad. Nauk SSSR 192 (1970), 963--965.
\bibitem{GurSav}
B. M. Gurevi\v{c} and S. V. Savchenko, Thermodynamic formalism for symbolic Markov chains with a countable number of states, Uspekhi Mat. Nauk, 53 (1998), 3--106.
\bibitem{Hochman}
M. Hochman, Isomorphism and embedding of Borel systems on full sets, Acta Applicandae Mathematicae: Volume 126 (2013), 187--201; Erratum.
\bibitem{Hofbauer}
F. Hofbauer, On intrinsic ergodicity of piecewise monotonic transformations with positive entropy. Israel J. Math. 34 (1979), 213–-237.
\bibitem{KatokIHES}
A. Katok,  Lyapunov exponents, entropy and periodic orbits for diffeomorphisms, Publ. Math. I.H.E.S. 51 (1980), 137--173.
\bibitem{KH} 
A. Katok, B. Hasseblatt, {\it An introduction to the Modern Theory of Dynamical Systems,} Cambridge University Press, 1985.
\bibitem{Kechris}
A. Kechris, Classical descriptive set theory, Graduate Text in Mathematics, 156, Springer, 1995.
\bibitem{Kitchens}
B. Kitchens, Symbolic dynamics. 
One-sided, two-sided and countable state Markov shifts. Universitext. Springer, 1998.
\bibitem{Krieger}
W. Krieger, On entropy and generators of measure-preserving transformations, Trans. Amer. Math. Soc. 149 (1970), 453--464.
\bibitem{Krieger2}
W. Krieger, On the subsystems of topological Markov chains, Ergodic Th.
Dynam. Systems, 2 (1982), 195--202.
\bibitem{LindMarcus}
D. Lind, B. Marcus, An introudction to symbolic dynamics and coding, Cambridge University Press, 1995.
\bibitem{Ornstein}
D. Ornstein,Bernoulli shifts with the same entropy are isomorphic,  Adv. Math. 4 (1970), 337--352.
\bibitem{Petersen}
K. Petersen, Ergodic theory. Cambridge Studies in Advanced Mathematics, 2. Cambridge University Press, 1983.
\bibitem{QS1}
A. Quas, T. Soo, Weak mixing suspension flows over shifts of finite type are universal, J. Mod. Dynam. 6 (2012), 427--449. 
\bibitem{QS2}
A. Quas, T. Soo, Ergodic universality of some topological dynamical systems, 	arXiv:1208.3501 [math.DS].
\bibitem{RHRHTU}
F. Rodriguez-Hertz, M.A. Rodriguez-Hertz, A. Tahzibi, R. Ures
Maximizing measures for partially hyperbolic systems with compact center leaves ,
Ergodic Th. Dynam. Systems 32 (2012), 825--839. 
\bibitem{Sarig}
O. Sarig, Symbolic dynamics for surface diffeomorphisms with positive entropy, J. Amer. Math. Soc. 26 (2013), 341--426.
\bibitem{Sarig2}
O. Sarig, Bernoulli equilibrium states for surface diffeomorphisms, J. Mod. Dynamics 5 (2011).
\bibitem{Serafin}
J. Serafin, Non-existence of a universal zero-entropy system, Israel J. Math. 194 (2013), 349--358.
\bibitem{Shelah}
S. Shelah, B. Weiss, Measurable recurrence and quasi-invariant measures, Israel J. Math. 43 (1982), 154-–160.
\bibitem{Takahashi}
Y. Takahashi, Isomorphisms of $\beta$-automorphisms to Markov automorphisms,
Osaka J. Math. 10 (1973), 175--184. 
\bibitem{Weiss1}
B. Weiss, Measurable dynamics, in: Conference in modern analysis and probability (New Haven, Conn., 1982),  
Contemp. Math. 26 (1984), 395–-421.
\bibitem{Weiss2}
B. Weiss, Countable generators in dynamics--universal minimal models, in: Measure and measurable dynamics (Rochester, NY, 1987),  Contemp. Math. 94 (1989), 321-–326.
\end{thebibliography}

\end{document}